\documentclass{amsart}
\usepackage{amssymb,amsfonts,amsmath,amsthm}
\usepackage[all]{xy}
\usepackage[shortlabels]{enumitem}

\newtheorem{theorem}{Theorem}[section]
\newtheorem{proposition}[theorem]{Proposition}
\newtheorem{question}[theorem]{Question}
\newtheorem{corollary}[theorem]{Corollary}
\newtheorem{definition}[theorem]{Definition}
\newtheorem{lemma}[theorem]{Lemma}
\newtheorem{example}[theorem]{Example}
\newtheorem{remark}[theorem]{Remark}
\newtheorem{claim}[theorem]{Claim}

\newtheorem*{theoremA}{Theorem A}
\newtheorem*{theoremB}{Theorem B}

\newcommand{\real}{\mathbb{R}}

\newcommand{\Si}{\Sigma}
\newcommand{\si}{\sigma}
\newcommand{\te}{\theta}

\newcommand{\al}{\alpha}
\newcommand{\om}{\omega}
\newcommand{\na}{\nabla}
\newcommand{\ep}{\epsilon}
\newcommand{\ga}{\gamma}
\newcommand{\Ga}{\Gamma}
\newcommand{\be}{\beta}
\newcommand{\de}{\delta}
\newcommand{\ti}{\tilde}
\newcommand{\vol}{\text{\rm vol}}

\newcommand{\lan}{\left\langle}
\newcommand{\p}{\partial}
\newcommand{\ran}{\right\rangle}

\newcommand{\g}{\gamma}
\newcommand{\m}{\mathcal}
\newcommand{\mr}{\mathring}

\title[Homotopy groups, focal points and totally geodesic immersions]{Homotopy groups, focal points and totally geodesic immersions}
\author{S\'ergio Mendon\c ca}
\address{Departamento de An\'alise, Instituto de Matem\'atica,
Universidade Federal Fluminense, Niter\'oi, RJ, CEP 24020-140,
Brasil} \email{sergiomendonca@id.uff.br}
\author{Heudson Mirandola}
\address{Departamento de M\'etodos Matem\'aticos, Instituto de Matem\'atica, Universidade Fe\-deral do Rio de Janeiro, Rio de Janeiro, RJ, CEP 21945-970, Brasil} \email{heudson@impa.br}

\thanks{This work is partially supported by CNPq.}

\begin{document}
\maketitle
\begin{abstract}  In this paper we consider on a complete Riemannian manifold $M$ an immersed totally geodesic hypersurface $\Si$ existing together with 
an immersed submanifold $N$ without focal points. No curvature condition is needed.
We obtained several connectedness results relating the topologies of $M$ and 
$\Si$ which depend on the codimension of $N$.
 \end{abstract}

\section{\bf Introduction}\label{introduction}

Unless otherwise stated, all manifolds in this paper will be assumed to be connected. Generalizing a result of Hadamard about the intersection of geodesics in a convex surface, Frankel \cite{Fr1} proved the following theorem.

\begin{theoremA}(Frankel \cite{Fr1}) Let $M^m$ be an $m$-dimensional closed Riemannian manifold of positive curvature,  $\Si$ a closed manifold and $f:\Si\to M$ a totally geodesic immersion with codimension at most $m/2$. Then the induced homomorphism $f_*^1:\pi_1(\Si)\to \pi_1(M)$ is surjective.
\end{theoremA}

There have been many generalizations of this theorem. Recently the relationship between the topology of totally geodesic submanifolds and the topology of ambient manifolds of positive curvature has been deepened in \cite{W} and \cite{FMR}.

The assumption of positive curvature is essential in the above results. In fact, given any manifold $M$ and any embedded closed submanifold $\Si$, one can construct metrics on $M$ such that $\Si$ is totally geodesic. Thus just the existence of a totally geodesic hypersurface does not imply any topological restrictions. Our idea here is to replace
the curvature hypothesis by the additional assumption of the existence of a complete isometric immersion $g:N\to M$ without focal points. Some examples in this introduction will illustrate this situation.

The work of Hermann (\cite{he}) and its generalization by Bolton (\cite{bo}) show that just the existence of a complete isometric immersion $g:N\to M$ without focal points strongly relates the topologies of $M$ and $N$ (see Theorem B and Corollary \ref{bolton2} below). Our expectation when we began this
work was that the union of the two hypotheses (existence of a totally geodesic hypersurface and a
submanifold without focal points) should also restrict the relation between the topologies of $M$ and $\Si$. In fact we obtain strong topological restrictions and even the compactness of $N$.

Let $M$ and $N$ be Riemannian manifolds and $g:N\to M$ an immersion. We will denote by $\m N_g$ and $\m N_g^1$, respectively, the normal bundle and the unit normal bundle of $g$. 
We recall the result of Bolton, which extends the work of Hermann (\cite{he}):

\begin{theoremB}(Bolton \cite{bo}) Let $M$ and $N$ be complete Riemannian manifolds and $g:N\to M$ an isometric immersion without focal points. Then the normal exponential map  $\exp^\perp:\m N_g\to M$ is a covering map.
\end{theoremB}
\begin{remark} \label{bolton} As a direct consequence of Theorem B it follows that if $M$ is simply-connected then $\exp^\perp:\m N_g\to M$ is a diffeomorphism, which implies that: (a) $g$ is an embedding and a homotopy equivalence, since $g=
\exp^\perp\circ j_0$ where $j_0:N\to \m N_g$ is the homotopy equivalence $j_0(x)=(x,0)$; (b) $N$ is simply-connected; (c) The boundary $S_\ep(g(N))$ of 
an $\ep$-tubular neighborhood of $g(N)$ is diffeomorphic to $\m N_g^1$. 
\end{remark}

Given manifolds $S, T$ and a smooth map $h:S\to T$ we will always denote by $\ti S, \ti T$ their universal covers 
with induced metrics, and by $\ti h:\ti S\to\ti T$ any lifting of $h$. Using the corresponding metrics we 
see that if an isometric immersion $g:N\to M$ is free of focal points then any lifting $\ti g:\ti N\to \ti M$  
has the same property. Applying Theorem B and Remark \ref{bolton} we have the  following

\begin{corollary} \label{bolton2} Let $g:N\to M$ be as in Theorem B. Then any lifting 
$\ti g:\ti N\to\ti M$ is an embedding and a homotopy equivalence. Furthermore 
the homomorphism $g_*^i:\pi_i(N)\to \pi_i(M)$ is  an isomorphism for $i\geq 2$ and a monomorphism  for $i=1$.
\end{corollary}

Now we recall some topological definitions. Consider path-connected topological spaces $A, B, C$ with $B\subset C$ and continuous maps $\al:A\to B$, $\iota:B\to C$ and
$\beta:A\to C$ where $\iota$ is the inclusion map and $\beta=\iota\circ\al$. It is said that $\al$ is $k$-connected if $\al_*^i:\pi_i(A)\to\pi_i(B)$ is an isomorphism for $1\le i\le k-1$
and an epimorphism for $i=k$. We say that the pair $(C,B)$ is $k$-connected if the inclusion $\iota$ is $k$-connected, and it is well known that this is equivalent to the fact that the relative homotopy groups $\pi_i(C,B)=0$
 \ for $1\le i\le k$. If $\al$ is a homeomorphism, the $k$-connectedness of $(C,B)$ is 
 clearly equivalent to the $k$-connectedness of $\be$. However, without this hypothesis, the $k$-connectedness of $\beta$ is not equivalent to the $k$-connectedness of the pair $(C,\al(A))$, even if $\al(A)=B$. For example, consider the inclusion of a $2$-torus $\iota:T\to \real^3$,
and consider the universal covering $\al:\real^2\to T$. Another example is to consider the universal covering
$\al:S^n\to\real P^n$ of the $n$-dimensional real projective space and the identity map $\iota:\real P^n\to\real P^n$. We say that
$C$ is $k$-connected if the pair $(C,p)$ is $k$-connected for some $p\in C$. In accordance with Remark \ref{bolton},  if $N$ and $M$ are complete manifolds 
and $g:N\to M$ is an 
isometric immersion without focal points, then the inclusion $\iota:S_\ep(\ti g(\ti N))\to \ti M$ is $(m-n-1)$-connected.

Our first result is the following.

\begin{theorem} \label{tgmth} Let $f:\Si^{m-1}\to M^m$ be a totally geodesic immersion of a closed manifold $\Si$ of finite fundamental group in a complete Riemannian manifold $M$ and $g:N^n\to M^m$ be an isometric immersion without focal points of a complete manifold $N$. Then $N$ is compact with finite fundamental group.  
Furthermore we have:
\begin{enumerate}[(1)]
\item\label{Ncompactafinita} if $m-n=1$ then $\Si$ and $N$ have diffeomorphic universal covers, any lifting $\ti f: \ti \Si\to \ti M$ is an embedding and a homotopy equivalence; 
furthermore the homomorphism $f_*^1:\pi_1(\Si)\to \pi_1(M)$ is a monomorphism;
\item\label{it3} if $m-n\ge 2$ then it holds that:
\begin{enumerate}[(a)]
\item \label{fgintersection}$f(\Si)\cap g(N)=\emptyset$;
\item $M$ is noncompact with finite fundamental group;
\item\label{ng1} $\Si$ and the normal unit bundle $\m N_g^1$ have diffeomorphic universal covers;
\item\label{it5} the homomorphism $\iota_*^1:\pi_1(f(\Si))\to \pi_1(M)$, induced by the inclusion $\iota:f(\Si)\to M$, is surjective;
\item \label{pi2infinito} if $m-n=2$ then $\pi_2(M)$ and $\pi_2(N)$ are infinite, $f_*^2:\pi_2(\Si)\to \pi_2(M)$ is a monomorphism and the quotient group $\pi_2(M)/\pi_2(\Si)=\Bbb{Z}$;
\end{enumerate}
\item\label{m-n3} if $m-n\geq 3$ then $f_*^1:\pi_1(\Si)\to \pi_1(M)$ is a monomorphism, any lifting $\ti f:\ti\Si\to \ti M$ is 
an embedding and $\ti f(\ti\Si)$ is diffeomorphic to $S_\ep(\ti g(\ti N))$ via a natural isotopy on $\ti M$, hence 
$\ti f:\ti\Si\to \ti M$ is $(m-n-1)$-connected.
\end{enumerate}
\end{theorem}

\begin{remark} \label{remarknaofalaremos} Note that Theorem \ref{tgmth} provides strong informations on $\ti f:\ti \Si\to \ti M$, which leads 
to immediate conclusions on the homomorphisms $f_*^i:\pi_i(\Si)\to \pi_i(M)$ for $i\ge 2$  
and we will not detail all of them here. For example,  if $f$ is an embedding and $m-n\ge 3$ we may use Items 
\ref{it3}-\ref{it5} and \ref{m-n3} to conclude that $f$ is $(m-n-1)$-connected.
\end{remark}

\begin{remark} \label{dimensions} In several points Theorem \ref{tgmth} may not be improved:
 \begin{enumerate} [(i)]
\item Example \ref{torus} below shows that the finiteness of
$\pi_1(\Si)$ is an essential assumption;
\item for the case $m-n=1$, Example \ref{example m-n1} below shows that $f(\Si)$ may intersect $g(N)$ and $M$ may be compact with infinite fundamental group (compare with Item \ref{it3}-\ref{fgintersection}); moreover, $f_*^1$ may be non-surjective even when $f$ is an embedding;
\item \label{remarknm2} for the case $m-n=2$, Example \ref{examplebundle} will show that the map $f_*^1$ could be non-injective or non-surjective;
\item \label{remarknm3}for the case $m-n\geq 3$, Example \ref{euclidean} will show that $f_*^{m-n-1}$ may be non-injective (compare with Remark \ref{remarknaofalaremos}); Example \ref{examplewarped} will show that $f_*^1$ may be non-surjective, and it also shows that the embeddedness of $f$ is essential to obtain the $(m-n-1)$-connectedness as in Remark 
\ref{remarknaofalaremos}.
\end{enumerate}
\end{remark}

We think it is interesting to consider the following question.
\begin{question} Assume the hypotheses of Theorem \ref{tgmth} with $m-n\ge 2$. Is it true that $g:N\to M$ is a homotopy equivalence?
\end{question}
By Corollary \ref{bolton2}, this question is equivalent to asking if $g^1_*:\pi_1(N)\to \pi_1(M)$ is
 an epimorphism. We notice that in the proof of Theorem 1.1 in [MM] this question was answered positively  in the very particular case that $N$ is a point. Example \ref{example m-n1} below shows that the answer would 
 be negative if we take $m-n=1$.

The next examples were cited in Remark \ref{dimensions}.
\begin{example}\label{example m-n1}
Consider the Riemannian product $M=N'\times S^1$ where $N'$ is a closed manifold with finite fundamental group and $S^1$ is a round circle. Note that $\Sigma=N' \times\{q\}$ has finite fundamental group. Take $N=\Si$ and let $f=g:N\to M$ be the inclusion map. The embedding $g$ is free of focal points and $f$ is totally geodesic. Thus Theorem \ref{tgmth} applies for $m-n=1$. Note that $f_*^1$ is not surjective, $f(\Si)\cap g(N)\not=\emptyset$ and $M$ is compact with infinite fundamental group.
\end{example}

\begin{example}\label{torus} Consider $M=N'\times T^k$ where $N'$ is a closed manifold with infinite
fundamental group and $T^k$ is a $k$-dimensional flat torus with $k\ge 2$. Let $f:\Sigma=N'\times T^{k-1}\times\{q\}\to M$ be the inclusion map. Notice that $f$ is a totally geodesic embedding. Choose $p\in T^{k-1}$. Let $g:N=N'\times \{(p,q)\}\to M$ be the inclusion map. The embedding $g$ is free of focal points with codimension $k\ge 2$. Note
 that $\pi_1(\Si)$ is infinite, hence Theorem \ref{tgmth} does not apply. In fact, several conclusions in this theorem  fail: $N$ has infinite
 fundamental group; $f(\Si)\cap g(N)\not=\emptyset$; $M$ is compact with infinite fundamental group; $f$ is an embedding and $f_*^1$ is not surjective.
\end{example}

\begin{example}\label{examplebundle}
Let $p:M^m\to N^n$ be a vector bundle over a manifold $N$ and fix a smooth fiber metric $x\in N\mapsto \left<\,,\,\right>_x$ where $\left<\,,\,\right>_x$ is an inner product on the fiber $V_x$. Let
 $S_x\subset V_x$ be the unit sphere centered at the origin and set $\m S_1=\cup_{x\in N}S_x$.
 Let $g:N\to g(N)=N_0\subset M$ be the null section.
 Proposition \ref{bundle} will show that there exists a Riemannian metric $\om$ on $M$ such that
 $\m S_1$ is a totally geodesic hypersurface of $M$ and $g$ is free of focal points. It
 holds also that $\exp^\perp:\m N_g\to M$ is a diffeomorphism. If further $N$ is compact then $M$ is
 complete.
 This example may be used to see that some conclusions in Theorem \ref{tgmth}
may not be improved (see Remark \ref{dimensions}-\ref{remarknm2}). In fact, we first
consider the particular case that $M$ is the tangent bundle $M=TS^2$ equipped with
the metric $\om$ and
$\Si=\m S_1=T_{1}S^2=\{(p,v)\in TS^2\bigm||v|=1\}$.
Let $f:\Si\to M$ be the totally geodesic inclusion map and $g:N\to N_0=N\times\{0\}\subset TS^2$  the
free of focal points null section. It is well known that $\Si$ is diffeomorphic to $SO_3$. We have that $\pi_1(\Si)=\pi_1(SO_3)=
\Bbb Z_2$ and $\pi_1(M)=\pi_1(TS^2)=0$.  Furthermore, it holds that $\pi_2(\Si)=\pi_2(SO_3)=0$ and $\pi_2(TS^2)=\Bbb Z$. Thus we conclude that $f_*^1$ is not injective and the quotient $\pi_2(M)/\pi_2(\Si)=\Bbb Z$ (in accordance with Item 
\ref{it3}-\ref{pi2infinito} in Theorem \ref{tgmth}).
 Now we consider the case that $(M,\om)=(T(\Bbb RP^2),\om)$
 and $\m S_1=T_1(\Bbb RP^2)=\{(p,v)\in T(\Bbb RP^2)\bigm||v|=1\}$. Let $f:\Si=SU_2\cong S^3\to \m S_1\subset M$ be the universal covering with the induced metric and $g: \real P^2\to \real P^2\times \{0\}\subset M$ the natural embedding.
Since $\pi_1(\Si)=0$ and $\pi_1(M)=\pi_1(T(\real P^2))=\Bbb Z_2$, the map $f_*^1$ is not surjective.
\end{example}

The following example is a typical situation where Theorem \ref{tgmth} holds and it will be used
in Example \ref{examplewarped}.

\begin{example}\label{euclidean} Define $M=\mathbb R^m$, with $m\ge 2$, endowed with the metric $ds^2=dr^2+\sigma^2(r)d\theta^2$ where $d\te^2$ is the standard metric on the unit sphere $S^{m-1}$ and $\si:[0,\infty)\to [0,\infty)$ is a smooth function satisfying $\si(r)>0$ for all $r>0$, $\si(0)=0$, $\si'(0)=1$ and $\si'(1)=0$. We know that $M$ is complete, the origin {\rm O} is a pole (in particular, $N={\{\rm{ O }\}}$ is free of focal points) and the sphere $\Si=S^{m-1}$  is totally geodesic. Notice that $\pi_{m-1}(\Si)=\Bbb Z$ and $\pi_{m-1}(M)=0$. Thus $f_*^{m-1}$ is not injective (see Remark \ref{dimensions}-\ref{remarknm3}).
\end{example}

\begin{example} \label{examplewarped} Take $B=(\mathbb{R}^k,ds^2)$, with $k\ge 3$, and $ds^2=dr^2+\sigma^2(r)d\theta^2$ being the metric introduced in Example \ref{euclidean}. Let $\m S=S^{k-1}\subset B$ be the totally geodesic unit sphere and $N'$ any Riemannian manifold. Consider a warped product $M=B\times_\rho N'$ and assume that the gradient $(\nabla \rho)|_{\m S}$ is tangent to $\m S$ (for example, take $\rho(x)=\si(|x|^2)$). The manifold $M$ is complete if $N'$ is complete (see Lemma 40 
in Chapter 7 of \cite {ON}). 
Let $g:N\to \{0\}\times N'$ and $f:\Si\to \m S\times N'$ be any covering maps. Proposition \ref{warped} below will show  that $g:N\to M$ is free of focal points and $f:\Si\to M$ is totally geodesic. If further $N'$ is compact with finite fundamental group, then Theorem \ref{tgmth} applies. Now take $N=N'=\real P^n$   and $\Si=\m S\times S^n$ with $n\geq 2$. Let $P:S^n\to \real P^n$ be the standard covering. Define the covering maps $f:\Si\to \m S\times N'$ given by $f(x,y)=(x,P(y))$ and $g:N\to \{0\}\times N'$ given by $g(z)= (0,z)$. The immersion $f:\Si\to M$ is not an embedding. The facts $\pi_1(\Si)=\{0\}$ and $\pi_1(M)=\pi_1(\real P^n)=\mathbb Z_2$ imply that $f_*^1$ is not surjective (see Remark \ref{dimensions}-\ref{remarknm3}). \end{example}

Let $S$ be an embedded submanifold of a Riemannian manifold $M$. Let $\m N_{S}$ be the normal bundle of $S$. We call an open subset $W\subset M$ an {\it $\epsilon$-tubular neighborhood} of $S$ if $W=\exp^\perp(\widetilde W)$, where $\widetilde W=\{(x,v)\in \m N_{S}\bigm||v|<\epsilon\}$ and the restriction $\exp^\perp|_{\widetilde W}$ is a diffeomorphism. Similarly we could define a {\it closed $\epsilon$-tubular neighborhood}. More generally we could define:

\begin{definition} \label{tubular} Let $V$ be a subset of $M$ that contains $S$. We say that $V$ is a {\it tubular neighborhood} of $S$ if $V=\exp^\perp(\widetilde V)$,    where $\exp^\perp|_{\widetilde V}$ is a diffeomorphism and $\widetilde V$ is a neighborhood (possibly with boundary)  of $\m N_{S}$ with the following property: if $(p,v)\in \widetilde V$ then $(p,tv)\in \widetilde V$ for all $t\in[0,1]$. 
\end{definition}

Comparing with Theorem \ref{tgmth}, the next result consider the assumption that $M$ is simply-connected instead the assumption that $\Si$ has finite fundamental group. By Remark \ref{bolton}, the simply-connectedness of $M$
implies that $N$ is simply-connected and $g$ is
 an embedding.
 
\begin{theorem} \label{tgmth2} Let $f:\Si^{m-1}\to M^m$ be a totally geodesic immersion of a closed manifold $\Si$ in a complete simply-connected Riemannian manifold $M$ and $g:N^n\to M^m$ a complete isometric embedding without focal points. Then $N$ is compact and the following conclusions hold:
\begin{enumerate}[(1)]
\item\label{m-n1}  if $m-n=1$ then $\Si$ is diffeomorphic to $N$, and $f$ is an embedding and a homotopy equivalence;
\item if $m-n=2$ then $\pi_1(\Si)$ is cyclic;
\item\label{m-n2} if $m-n\ge 2$ then $f(\Si)\cap g(N)=\emptyset$ and there exists a covering map $F:\Si\to S_\ep\cong\m N_g^1$ such that $f$ is smoothly homotopic the $j\circ F$, where $j:S_\ep\to M$ is the inclusion map; if further $f$ is an embedding then the following statements hold:
\begin{enumerate}[(a)]
\item $\bar F=F\circ f^{-1}:f(\Si)\to S_\ep$ is a diffeomorphism via an ambient isotopy, hence $f$ is $(m-n-1)$-connected;
\item $M-f(\Si)=A\sqcup B$ where the closure $\bar A$ is a compact tubular neighborhood of $g(N)$ with smooth boundary $f(\Si)$, and $\bar B$ is an unbounded smooth manifold with boundary $f(\Si)$.
\end{enumerate}
\item\label{m-ngeq3} if $m-n\geq 3$ then $f$ is an embedding and $\Si$ is simply-connected.
\end{enumerate}
\end{theorem}

By considering the exact sequence of the fiber bundle $S^{m-n-1}\to S_\ep\to g(N)$ it is easy to obtain from 
Theorem \ref{tgmth2} the following
\begin{corollary} Under the hypotheses of Theorem \ref{tgmth2} if $m-n=2$ then $f_*^i:\pi_i(\Si)\to \pi_i(M)$ is a monomorphism  for $i=2$ and an isomorphism for $i\ge 3$.
\end{corollary}

\begin{remark} Theorem \ref{tgmth2} shows that if $m-n\neq 2$ then $f$ is an embedding and $\Si$ is simply-connected. However, in the case $m-n=2$ both conditions may fail. In fact, if
we consider in Example \ref{examplebundle} the special case that  $\Si=SU_2\cong S^3$ and $f:\Si\to \m S_1=T_1 S^2$ is the universal
 covering, we see that  the map $f:\Si\to M=TS^2$ is not an embedding. Still in Example \ref{examplebundle}, the case that $\Si=\m S_1=T_1 S^2$ and $f$ is the identity map we see that $\Si$ is not simply-connected, showing that the simply-connectedness of $\Si$ may not occur in codimension two.
\end{remark}

For the next result we need the following definition.

\begin{definition} \label{graph} Let $S$ be a submanifold of a Riemannian manifold $M$ and 
consider $X\subset M$ with the induced topology. We say that $X$ is a normal graph over $S$ if there exists a homeomorphism $h:S\to X$ such that for any point $x\in S$ there exists a unique unit speed geodesic which starts at $x$ orthogonally to $S$ and ends at $h(x)$. \end{definition}

Comparing with Theorem \ref{tgmth2}, the next result replaces the assumption that $\Si$ is compact by the condition that $\Si$ is
properly embedded in $M$ with $g(N)\cap \Si=\emptyset$. No codimension
condition on $N$ is needed.
\begin{theorem}\label{tgth2} Let $\Si$ be a properly embedded totally geodesic hypersurface in a complete simply-connected manifold $M$. Let $g:N\to M$ be a complete isometric embedding without focal points with $g(N)\cap \Si=\emptyset$. Then we have:
\begin{enumerate}
\item for any $\ep>0$, the hypersurface $\Si$ is a normal graph over an open subset of the boundary of the closed $\epsilon$-tubular neighborhood of $g(N)$;
\item $M-\Si=A\sqcup B$, where the closure $\bar A$ is a (possibly unbounded) tubular neighborhood of $g(N)$ with smooth boundary $\Si$ and $\bar B$ is an unbounded smooth manifold with boundary $\Si$;
\item for each point $x\in \bar B$, the unique unit speed geodesic which starts orthogonally at $g(N)$ and ends at $x$ intersects $\Si$ transversely at a unique point.
\end{enumerate}
\end{theorem}

Theorem B shows that the hypothesis that $\exp^\perp:\m N_g\to M$ is a covering map is weaker than the assumption that $M$ and $N$ are complete with $g$ free of focal points. The next result presents a situation where these two assumptions are equivalent.
\begin{theorem} \label{covering}  Let $f:\Si^{m-1}\to M^m$ be a totally geodesic immersion of a closed manifold $\Si$ in a Riemannian manifold $M$. Let $g:N^n\to M^m$ be an immersion. Assume that one of the following conditions holds:
\begin{enumerate}[(i)]
\item\label{diff} $\exp^\perp:\m N_g\to M$ is a diffeomorphism;
\item\label{finite} $\exp^\perp:\m N_g\to M$ is a covering map and $\Si$ has finite fundamental group.
\end{enumerate} Then $M$ is complete and  $N$ is compact, hence Theorem \ref{tgmth} applies if Item \ref{finite} is satisfied. If Item \ref{diff} holds we have that:
\begin{enumerate}[(1)]
\item \label{cover} if $m-n=1$ then $\rho\circ f:\Si\to g(N)$ is a covering map, where $\rho:M\to g(N)$ is the natural strong 
deformation retraction  given by $\rho(\exp^\perp(x,v))=g(x)$;
\item if $m-n\geq 2$ then the same conclusions in Item \ref{m-n2} of Theorem \ref{tgmth2} hold.
\end{enumerate}
\end{theorem}

The following example illustrates a situation where Item \ref{cover} in Theorem \ref{covering} applies.
\begin{example} We consider the complete flat Moebius strip \begin{center} $M=\left([-1,1]\times \real\right)/{\sim}\,,$ \ where $(-1,t)\sim (1,-t)$ for all $t\in\real$.\end{center} Denote by
$\bar\al\in M$ the class of $\al$. Take
$$\Si=\left\{\overline{(x,t)}\in M \bigm| x\in [-1,1],\ t= 1 \mbox{ or } t=-1\right\}$$
and $N=\left\{\overline{(x,0)}\in M \bigm| x\in [-1,1]\right\}$ and let $f:\Si\to M$ and
$g:N\to M$ be the inclusion maps.
It is easy to see that $f$ is totally geodesic and that $\exp^\perp:\m N_g \to M$ is a diffeomorphism. \end{example}

Theorem \ref{covering} suggests that Theorems \ref{tgmth}, \ref{tgmth2} and \ref{tgth2}
could be rewritten in technical more general versions (see Theorems \ref{gth}, \ref{12linha}  and \ref{completeness2}, where
we also generalize the totally geodesic condition).

The rest of this paper is organized as follows. In section 2 we prove Theorem \ref{tgth2} and in section 3 we prove Theorems \ref{tgmth}, \ref{tgmth2} and \ref{covering}.  In section 4 we present proofs for facts present in some examples in the introduction.

\begin{remark} Since this paper uses Proposition 4.1 in \cite{mz}, the first author would like to use this occasion to inform that it cames to his knowledge that Theorem A in \cite{mz} is implied by a more general result present in the doctor thesis of Flor\^encio F. Guimar\~aes.
\end{remark}

{\bf Aknowledgement.} The authors would like to thank Professor Detang Zhou for useful discussions during 
the preparation of this paper. They also thank the referee for very interesting suggestions.

\section{\bf Proof of Theorem \ref{tgth2}}\label{prooftgth2}

Theorem \ref{tgth2} follows from Theorem B and the next technical general theorem  (compare with Theorem 1.2 in \cite{mm}):

\begin{theorem}\label{gth} Let $\Si$ be a properly embedded hypersurface in a Riemannian manifold $M$. Let $g:N\to M$ be an embedding such that $\exp^\perp: \m N_g\to M$ is a diffeomorphism. Assume that $g(N)\cap \Si=\emptyset$ and that the unit speed geodesics tangent to $\Si$ do not intersect $g(N)$ orthogonally. Then we have:
\begin{enumerate}
    \item\label{it1gth} given $\ep>0$, the hypersurface $\Si$ is a normal graph over an open subset $\Omega$ of the boundary $S_\ep$ of the closed $\epsilon$-tubular neighborhood of $g(N)$;
        \item \label{it2gth} $M-\Si=A\sqcup B$, where $\bar A$ is a (possibly unbounded) tubular neighborhood of $g(N)$ with smooth boundary $\Si$ and $\bar B$ is an unbounded smooth manifold with boundary $\Si$ and contained in $\rho^{-1}(\rho(\Si))$, where 
        $\rho:M\to g(N)$ is the natural projection given by $\rho(\exp^\perp (q,v))=g(q)$;
    \item\label{it3gth} for each point $x\in \bar B$ the unique unit speed geodesic $\g_x:[0,+\infty)\to M$ which starts orthogonally at $g(N)$ and passes through $x$ intersects $\Si$ transversely at a unique point;
    \item\label{isotopy} the map $F:\Si\to S_\ep$ given by $F(x)=\g_{x}(\ep)$ is a diffeomorphism onto $\Omega$ via a natural isotopy on $\rho^{-1}(\rho(\Si))$.    
\end{enumerate}
\end{theorem}
It is a little surprising that the hypotheses of completeness of $M$ and $N$ are not needed in Theorem \ref{gth}.
What compensates this weakness is the fact that for each point $x$ in $M-g(N)$ there exist a neighborhood $W$ of $x$ with the
property that $\g_y((0,+\infty))\subset W$, for all $y\in W$.
Thus we can use local arguments to explore the fact that $\Si$ is properly embedded.

The two simple examples bellow illustrate this theorem.
\begin{example} Consider the surface $\Si\subset \real^3$ of equation $z=\frac{1}{x^2+y^2}$. Notice that $\Si$ is properly embedded in $M=\{(x,y,z)\in \real^3 \bigm| z>0\}$. Let $g:N=\{(0,0,z)\in \real^3\bigm| z>0\}\to M$ be the inclusion map. Notice that $\exp^\perp:\m N_g\to M$ is a diffeomorphism and the unit speed geodesics tangent to $\Si$ do not intersect $N$ orthogonally. Thus Theorem \ref{gth} applies although $M$ and $N$ are not complete.
\end{example}
\begin{example} Consider the surface $\Si\subset \real^3$ of equation $z=\frac 1{1-x^2-y^2}$, with $x^2+y^2<1$. The
surface $\Si$ is properly embedded in $M=\real^3$. Let $N$ be the $xy$-plane and $g:N\to M$ the
inclusion map. Theorem \ref{gth}
applies. Notice that $$B=\left\{(x,y,z)\in\real^3\bigm|x^2+y^2<1 \mbox{ and } z> \frac 1{1-x^2-y^2}\right\},$$ and that the complement $A=\real^3-\bar B$ is a tubular neighborhood of $N$ in the sense of Definition \ref{tubular}.
\end{example}

Before proving Theorem \ref{gth} we would like to present some notations and a very simple result that will be used in several places in this paper. Let $g:N^n\to M^m$ be an immersion such that $\exp^\perp:\m N_g\to M$ is a diffeomorphism. Then $g$ is an embedding and the natural projection $\rho:M\to g(N)$  is a fiber bundle.  From the fact that $\exp^\perp$ is a diffeomorphism it follows easily that $\rho$ is a homotopy equivalence between $M$ and $g(N)$. Furthermore, for each point $x\in M-g(N)$ there exists a unique unit speed geodesic $\g_x:[0,+\infty)\to M$ containing $x$ which intersects $g(N)$ orthogonally at $t=0$ satisfying $\g_x(d_x)=x$ for some $d_x>0$. If $x\in g(N)$ we set $d_x=0$. We should notice that we don't know here if $d_x$ is the distance $d(x,g(N))$ (remember that we are not assuming that $M$ and $N$ are complete). Fix $\ep>0$ and let $S_\ep$ be the boundary of the $\ep$-tubular neighborhood of $g(N)$. Let $j:S_\ep\to
M$ be the inclusion map.
\begin{lemma} \label{FG} Let $g:N^n\to M^m$ be an embedding such that $\exp^\perp:\m N_g\to M$ is a diffeomorphism. Assume the notations above. Let $f:\Si^{m-1}\to M^m$ be an immersion such that the unit speed geodesics tangent to $f(\Si)$ do not intersect $g(N)$ orthogonally. Then we have:
\begin{enumerate}[(a)]
\item \label{dx} The map $\xi: M\to \real$ given by $\xi(x)= \frac 12(d_x)^2$ is smooth on $M$;
\item \label{F} If $f(\Si)\cap g(N)=\emptyset$ then the map $F:\Si\to S_\ep$ given by $F(p)=\ga_{f(p)}(\ep)$ is a local diffeomorphism and $f$ is smoothly homotopic to $(j\circ F)$;
\item \label{G} If $m-n=1$ then the map $G:\Si\to g(N)$ given by $G=\rho\circ f$ is a local diffeomorphism.
\end{enumerate}
\end{lemma}
\begin{proof}
We first prove Item \ref{dx}. For $x\in M$, we  write $(p,v)=(\exp^\perp)^{-1}(x)$ and define  $P:x\mapsto v$. Since $d_x=|P(x)|$, we conclude that $d_x$ is smooth on $M-g(N)$ and 
the map $\xi(x)=\frac 12|P(x)|^2$ is smooth on $M$. 

To prove Item \ref{F} define the smooth vector field $X(x)$ as the gradient $\na\xi(x)$, for 
any $x\in M$. Clearly we have that $X(x)=d_x\,\g_{x}'(d_x)$, if $x\notin g(N)$ and $X(x)=0$ if
$x\in g(N)$. Fix a small open subset $U\subset \Si$ such that $f|_U:U\to M$ is an embedding.
The orbits of $X$ are orthogonal to $S_\ep$ and, by hypothesis, transversal to $f(U)$. Thus, reducing
$U$ if necessary, it is not difficult to see that the flow of $X$ gives a standard diffeomorphism $\varphi:f(U)\to V$ where $V$ is a small
open subset of $S_\ep$. Note that $F|_U=\varphi\circ f|_U$, which implies that $F|_U: U\to V$ is a diffeomorphism, hence
$F$ is a local diffeomorphism.
The map $H:[0,1]\times\Si\to M$ given by
$$H(t,x)=\g_{f(x)}\left((1-t)\ep+td_{f(x)}\right)$$ provides a smooth homotopy between $f$ and $j\circ F$.

To prove Item \ref{G},  we assume that $m-n=1$. Given  $z\in M$, the fiber $\rho^{-1}(\{\rho(z)\})$ coincides 
with the image of a geodesic $\be_z:\real\to M$ which intersects
$g(N)$ orthogonally and satisfies $\be_z(0)=z$. If further there exists $x\in \Si$ such that $z=f(x)$, we know by hypothesis that $\be_z'(0)\notin df_x(T_x\Si)$, hence $f$ is transversal to $\be_z$. Thus the
map $G:\Si\to g(N)$ given by $G(x)=\rho\circ f(x)$ is a submersion, hence a local diffeomorphism.
\end{proof}

\begin{proof}[Proof of Theorem $\ref{gth}$] Since $\exp^\perp:\m N_g\to M$ is a diffeomorphism and
$\Si\cap g(N)=\emptyset$, Lemma \ref{FG} implies that the
map $F:\Si\to S_\ep$ given by $F(p)=\g_{p}(\ep)$ is
a local diffeomorphism onto its open image $\Omega=F(\Si)\subset S_\ep$.
  To prove that $F$ is a diffeomorphism it is sufficient to show that $F$ is injective. In order to show this fact we define the set
\begin{equation*}
\mathcal C=\left\{p\in \Si \ \left| \mbox{ the cardinality } \#\bigl(\g_p([0,d_p])\cap \Si\bigr)=1 \right.\right\}.
\end{equation*}
We just need to prove that $\mathcal C=\Si$.

\begin{claim} $\mathcal C\neq \emptyset.$ \end{claim}

In fact we take $p\in\Si$. Since $\Si$ is properly embedded it follows that $\ga_p([0,d_p])\cap \Si$ is a compact set. By using the facts that $\ga_p$ intersects $\Si$ transversely and that $\Si$ is properly embedded, we obtain that $\ga_p([0,d_p])\cap \Si$ is a discrete set,
and thus it is finite. Thus we write $$\ga_p([0,d_p])\cap\Si=\{\ga_p(t_1),\ldots,\ga_p(t_k)\}$$ with $t_1<\ldots<t_k=d_p$. Notice that $t_1>0$ since $\Si\cap g(N)=\emptyset$. Set $p_1=\g_p(t_1)$. Since $\ga_{p_1}|_{[0,d_{p_1}]}=\ga_p|_{[0,t_1]}$, we
 have that \ $\# \bigl(\Si\cap \ga_{p_1}([0,d_{p_1}])\bigr)=1$, hence $p_1\in\mathcal C$, which concludes the proof of this Claim.

\begin{claim} $\Si-\mathcal C$ is open as a subset of $\Si$.
\end{claim}

To prove this take $x_1\in \Si-\mathcal C$. So there exists $x_2\in \Si$ with $x_2\neq x_1$ and $x_2=\g_{x_1}(t)$ for some $0<t<d_{x_1}$. In particular $F(x_1)=F(x_2)=\g_{x_1}(\ep)$. Since $F$ is
a local diffeomorphism, there exist disjoint neighborhoods of $x_1$ and $x_2$ in $\Si$ mapped by $F$ onto the same neighborhood of $\g_{x_1}(\ep)$ in $S_\ep$. Thus we conclude that $\Si-\mathcal C$ is open in $\Si$.

\begin{claim} \label{closed} $\Si-\mathcal C$ is closed as a subset of $\Si$.
\end{claim}
In fact take a sequence $x_k\to x\in \Si$ with $x_k\in \Si-\mathcal C$.
By Lemma \ref{FG} we have that $d_{x_k}\to d_x$.  Since $\g_y(0)=\rho(y)$, where
$\rho:M\to g(N)$ given by $\rho(\exp^\perp (q,v))=g(q)$ is the natural projection, we have by continuity
of $\rho$ that $\g_{x_k}(0)\to \g_x(0)$.
Since $\Si$ is properly embedded there exists an open neighborhood $U$ of $x$ in $M$ such that the intersection $\Si\cap U$ is a topological disk and the restriction $F|_{U\cap \Si}$ is a diffeomorphism onto its open image. By passing to a subsequence we may assume that $\g_{x_k}'(0)\to v$. Since
   $x_k=\g_{x_k}(d_{x_k})=\exp^\perp\bigl(\g_{x_k}(0),d_{x_k}\g_{x_k}'(0)\bigr)$, by taking limits we obtain that $x=\exp^\perp\bigl(\g_{x}(0),d_{x}v\bigr)$. On the other hand we have that $x=\g_x(d_x)=\exp^\perp\bigl(\g_{x}(0),d_{x}\g_{x}'(0)\bigr)$, thus the
   injectivity of $\exp^\perp$ implies that $v=\g_x'(0)$. Since  $\g_{x_k}(0)\to \g_x(0)$ and $\g_{x_k}'(0)\to \g_{x}'(0)$ we obtain that $\ga_{x_k}\to \ga_x$ uniformly on compact sets. Since $x_k\notin \mathcal C$ there exists a point $y_k\neq x_k$ with $y_k\in\Si$ and $y_k=\g_{x_k}(t_k)$ with $0<t_k<d_{x_k}$. The sequence $(t_k)$ is bounded since $(d_{x_k})$ converges. Again by passing to a subsequence we can suppose that $t_k$ converges to some $t_0\in [0,d_x]$, hence $y_k=\g_{x_k}(t_k)\to x_0=\g_{x}(t_0)$,  which belongs to $\Si$ because $\Si$ is properly embedded. For large $k$ the point $x_k\in U$. Since $F|_{U\cap \Si}$ is injective and $F(x_k)=F(y_k)$ we have that $y_k\not\in U$ for large $k$, hence $x_0=\g_{x}(t_0)\neq x$. Thus $t_0<d_x$, hence $x\in \Si-\mathcal C$. This concludes
   the proof of Claim \ref{closed}.

By the connectedness of $\Si$ we conclude that $\mathcal C=\Si$ which proves  the following
\begin{claim} \label{Fdiffeo} $F:\Si\to \Omega$ is a diffeomorphism.
 \end{claim}

 From Claim \ref{Fdiffeo} we have that $\Si$ is a normal graph over the open subset $F(\Si)\subset S_\ep$ (see Definition \ref{graph}). This proves Item (\ref{it1gth}) in Theorem \ref{gth}.

Now we will prove that $\Si$ is the boundary of a tubular neighborhood of $g(N)$. Define the set
$$ A=\left\{x\in M \bigm | \g_x([0,d_x])\cap \Si=\emptyset \right\}.$$
\begin{claim} $A$ is an open subset of $M$. \end{claim}
      In fact it suffices to prove that $M-A$ is closed. Consider a sequence $x_k\in M-A$ such that $x_k\to x$. Thus for each $k$ there exists $y_k\in \g_{x_k}([0,d_{x_k}])\cap \Si$ and we write $y_k=\g_{x_k}(t_k)$, with $0<t_k\le d_{x_k}$. As in the proof of Claim \ref{closed}, we obtain by passing to a subsequence that $t_k\to t_0\in [0,d_x]$, hence $y_k=\g_{x_k}(t_k)\to y=\g_{x}(t_0)$. Since $\Si$ is properly embedded we have that $y\in \Si$, hence $y\in \g_x([0,d_x])\cap \Si$. Thus we conclude that $x\in M-A$, hence $M-A$ is closed.

\begin{claim} \label{claimboundary} $\bar A-A=\Si$.
\end{claim}

In fact, since $\mathcal C=\Si$, given any $p\in\Si$ we have that $\g_p([0,d_p])-\{p\}\subset A$. Thus $\Si\subset \bar A$. Clearly $\Si\cap A=\emptyset$ hence $\Si\subset \bar A-A$.  Now take $p\in \bar A-A$ and assume by contradiction that $p\not\in \Si$. Since $p\not\in A$ we have
$\g_p([0,d_p])\cap \Si\not=\emptyset$. Since the map $F$ is injective we obtain that $\#(\g_p([0,d_p])\cap \Si)\le 1$, hence
    $\g_p([0,d_p])$ intersects $\Si$ transversely at a unique point $q\in\Si$. Thus it holds that $\ga_q([0,+\infty))=\ga_p([0,+\infty))$ and $d_q<d_p$. Let $U$ be a small neighborhood of $q$ in $\Si$. For a small $0<\delta<d_p-d_q$ consider the set $$W=\{\g_x(t)\bigm|x\in U\ \mbox{ and }\ d_p-\delta<t<d_p+\delta\}.$$
Notice that $p=\g_q(d_p)$, hence $p\in W$. Since $\exp^\perp$ is a diffeomorphism the set $W$ is an open neighborhood of $p$.
By taking $U$ sufficiently small and using that $d_q<d_p-\de$, we obtain by continuity that $d_x<d_p-\de$ for all $x\in U$. Now
take $y\in W$. Then there exists $x\in U$ and $d_p-\de<t<d_p+\de$ such that $y=\g_x(t)=\g_x(d_y)$. Since $d_x<d_p-\de<t=d_y$ we have that
$x\in \g_y([0,d_y))\cap\Si$, hence $y\notin A$. Thus we have that $W\subset M-A$ which contradicts the fact that $p\in \bar A$. This concludes the proof of Claim \ref{claimboundary}.

Let us prove that $\bar A$ is a manifold with smooth boundary $\Si$.
Take a point $p\in \Si$. For a small neighborhood $V$ of $p$ in $\Si$ and small $\ep>0$, consider the set
$$\Gamma=\left\{\g_x(t)\bigm| x\in V;\ d_x-\ep <t\leq d_x\right\}.$$
By Lemma \ref{FG} we have that $d_x$ depends smoothly on $x$ on $M-g(N)$. Since $\exp^\perp$ is a diffeomorphism we conclude  that $\Ga$ is a parameterized neighborhood of $p$ in $\bar A$. Furthermore we have
$$\Ga=\left\{\g_x(t)\bigm| x\in V;\ d_x-\ep <t< d_x\right\}\cup\left\{\g_x(t)\bigm| x\in V;\ t= d_x\right\}=(A\cap\Ga)\cup (\Si\cap\Ga).$$
  Thus $\bar A$ is a smooth manifold with boundary $\Si$. We have proved that $\g_x([0,d_x])\subset A$ for all $x\in A$ and
  $\g_x([0,d_x])\subset \bar A$ for all $x\in \bar A$. Thus we conclude that $A$ and $\bar A$ are tubular neighborhoods of $g(N)$.

 To see that $\Si$ disconnects $M$ consider $B=M-\bar A$. We have $M=\bar A\sqcup B=A\sqcup\Si\sqcup B$. In particular we have $M-\Si=A\sqcup B$. Fix $p\in \bar B$. Since $p\notin A$ we conclude by the proof of Claim \ref{claimboundary}
 that $\ga_p$ intersects $\Si$ transversely at a unique point. Thus we have that
 $$\bar B=\left\{p\in M \bigm | \# (\g_p([0,d_p])\cap \Si)=1 \right\}=\left\{\g_x(t)\bigm| x\in \Si \mbox{\ and\ }t \ge d_x\right\},$$ which
 proves that $B$ is a connected non-compact manifold with boundary $\Si$. It is easy to see that 
 $\bar B$ is contained in $\rho^{-1}(\rho(\Si))$, since $F(\bar B)=F(\Si)=\Omega$.  
 Items (\ref{it2gth}) and (\ref{it3gth}) in Theorem \ref{gth} are proved.
 
 As in the proof of Lemma \ref{FG}, we consider the vector field $X(x)=d_x\g_x'(d_x)$ if $x\notin g(N)$ 
 and $X(x)=0$ if $x\in g(N)$.  
 Item (\ref{isotopy}) follows easily from the facts that the orbits of $X$ 
 intersect $\Omega$ and $\Si$ transversely (according with Item (\ref{it3gth})) and that $\rho^{-1}(\rho(\Si))$ is 
 invariant under the flow of $X$. 
\end{proof}

\section{\bf Proof of Theorems  \ref{tgmth}, \ref{tgmth2} and \ref{covering}}\label{proofthms}

To prove Theorems \ref{tgmth}, \ref{tgmth2} and \ref{covering} we first need to show the following simple
lemmas.
\begin{lemma} \label{orthogon} Let $f:\Si\to M$ be a totally geodesic immersion of a closed manifold $\Si$ in a Riemannian manifold $M$ and $g:N\to M$ an immersion. Assume that either
\begin{enumerate}[(a)]
\item $\exp^\perp:\m N_g\to M$ is a diffeomorphism; or
\item $\exp^\perp:\m N_g\to M$ is a covering map and $\Si$ has finite fundamental group.
\end{enumerate}
Then the unit speed geodesics of $M$ tangent to $f(\Si)$ do not intersect $g(N)$ orthogonally.
\end{lemma}
\begin{proof} First, we will assume that $(a)$ is satisfied. Assume by contradiction that there exists a unit speed geodesic $\g:[0,+\infty)\to M$ that starts orthogonally to $g(N)$ and is tangent to $f(\Si)$. Since $f$ is totally geodesic we conclude that $\ga([0,+\infty))\subset f(\Si)$, hence $\ga$ is bounded, since $\Si$ is
 compact. We write $\g(t)=\exp^\perp(q,tv)$, $t\geq 0$, for some $(q,v)\in \m N_g$ with $|v|=1$. Since $\exp^\perp$ is a diffeomorphism we have that $\ga$ is unbounded, which gives us a contradiction.

Now we assume (b). Consider on $\m N_g$ the metric induced by $\exp^{\perp}=\exp^\perp_g$ and define $\hat g:N\to \m N_g$ given by $\hat g(x)=(x,0)$. Let $\nu:\hat\Si\to \Si$ be the universal covering of $\Si$ with the induced metric. By the Fundamental Lifting Theorem, $f$ admits a lifting $\hat f:\hat \Si\to \m N_g$ such that the diagram below is commutative.
\begin{equation}\label{hat}
\xymatrix{
\hat\Si\ar[rr]^{\hat f}\ar[d]_{\nu} && \m N_g\ar[d]_{\exp^\perp_g}&& \\
\Si\ar[rr]^f                        &&  M &&  N\ar[ll]_g\ar[ull]_{\hat g}}\\
\end{equation}

Since $\Si$ is compact with finite fundamental group it follows that $\hat \Si$ is compact. Note that $\hat f$ is totally geodesic. Any unit speed geodesic $\g$ on $M$ starting orthogonally from $g(N)$ is lifted by $\exp^\perp_g$ to
a curve $\hat\g=\hat\g_{(q,v)}:[0,+\infty)\to \m N_g$ given by $\hat\g(t)=(q,tv)$, for some
$(q,v)\in\m N_g$ with $|v|=1$. Since we are considering on $\m N_g$ the metric induced by $\exp^\perp_g$, we
 have that the curves $\hat\g_{(q,v)}$ are the unit speed geodesics of $\m N_g$ that start orthogonally from $\hat g(N)=N\times\{0\}$, hence $\exp_{\hat g}^\perp:\m N_{\hat g}\to \m N_g$ is a diffeomorphism. Thus we may use Item (a) to conclude that the unit speed geodesics tangent to $\hat f(\hat \Si)$ do not intersect $\hat g(N)$ orthogonally. Since $\exp^\perp_g$ is a local isometry we conclude that the unit speed geodesics tangent to $f(\Si)$ do not intersect $g(N)$ orthogonally.
\end{proof}

\begin{lemma} \label{intersect} Let $f:\Si\to M$ and $g:N\to M$ be immersions in a Riemannian manifold $M$ with $\dim(\Si)>\dim(N)$. Assume that the unit speed geodesics of $M$ tangent to $f(\Si)$ do not intersect $g(N)$ orthogonally. Then $f(\Si)\cap g(N)=\emptyset$.
\end{lemma}
\begin{proof} Assume by contradiction that $f(\Si)\cap g(N)\not=\emptyset$. Then there exist $p\in \Si$ and $q\in N$ with $f(p)=g(q)$. Set $$V=\bigl(dg_q(T_qN)\bigr)^\perp\cap \bigl(df_p(T_{p}\Si)\bigr),$$ where $d f_{ p}$ denotes the derivative of $f$ at
$p$ and $T_{p}\Si$ the tangent space. Then we have:
\begin{eqnarray*}\dim(V)&\ge&\dim\bigl(dg_q(T_qN)\bigr)^\perp+\dim\bigl(df_p(T_{p}\Si)\bigr)-\dim(M)\\&=& (\dim(M)-\dim(N))+\dim(\Si)-\dim(M)\geq 1.
\end{eqnarray*}
Thus we can take $w\in V$ with $|w|=1$. Consider the geodesic $\g:\mathbb R\to M$ satisfying $\ga(0)=g(q)$ and $\g'(0)=w$. Thus $\ga$ is a geodesic tangent to $f(\Si)$ and orthogonal to $g(N)$, which is a contradiction.
\end{proof}

From Lemmas \ref{orthogon} and \ref{intersect} we obtain the following
\begin{corollary} \label{lastclaim} Let $f:\Si\to M$ be a totally geodesic immersion of a closed manifold $\Si$ in a Riemannian manifold $M$ and $g:N\to M$ an immersion with dimension $\dim(\Si)>\dim(N)$. Assume that either
\begin{enumerate}[(a)]
\item $\exp^\perp:\m N_g\to M$ is a diffeomorphism; or
\item $\exp^\perp:\m N_g\to M$ is a covering map and $\Si$ has finite fundamental group.
\end{enumerate} Then $f(\Si)\cap g(N)=\emptyset$.
\end{corollary}

\begin{remark} Example \ref{example m-n1} shows that the conclusion of Corollary \ref{lastclaim} fails\ if\
$\dim(\Si)=\dim(N)$.
\end{remark}

The next result will be important to prove Theorems \ref{tgmth}, \ref{tgmth2} and \ref{covering}.
\begin{theorem}\label{completeness} Let $f:\Si^{m-1}\to M^m$ be an immersion of a closed manifold $\Si$ in a Riemannian manifold $M$. Let $g:N\to M$ be an embedding such that $\exp^\perp:\m N_g\to M$ is a diffeomorphism. Assume that the unit speed geodesics of $M$ tangent to $f(\Si)$ do not intersect $g(N)$ orthogonally. Then $M$ is complete,  $N$ is compact and the following conditions hold:
\begin{enumerate}[(1)]
\item \label{covering2} if $m-n=1$ then $\rho\circ f:\Si\to g(N)$ is a covering map, where $\rho:M\to g(N)$ is the 
natural strong deformation retraction $\rho(\exp^\perp(x,v))=g(x)$;
\item \label{mn>2} if $m-n\geq 2$ then \ $f(\Si)\cap g(N)=\emptyset$, the map $F:\Si\to S_\ep\cong\m N_g^1$ given by $F(p)=\g_{f(p)}(\ep)$ is 
 a covering map and $f$ is smoothly homotopic to $j\circ F$, where $j:S_\ep\to M$ is the inclusion map; if further  
 $f$ is an embedding we have:
     \begin{enumerate}[(a)]
     \item\label{embedding} $\bar F=F\circ f^{-1}:f(\Si)\to S_\ep$ is a diffeomorphism 
     via an ambient isotopy, hence $f$ is $(m-n-1)$-connected; 
     \item\label{disjointunionsplit}
     $M-f(\Si)=A\sqcup B$ where the closure $\bar A$ is a compact tubular neighborhood of $g(N)$ with boundary $f(\Si)$, and $\bar B$ is an unbounded manifold with boundary $f(\Si)$.
    \end{enumerate}
\end{enumerate}
\end{theorem}

Under the hypotheses of Theorem \ref{completeness}, let $\iota:g(N)\to M$ be the inclusion map.
We know that $\iota\circ\rho$ is homotopic to the identity map on $M$. Thus if $m-n\ge 2$ for all $i$ we have  that
\begin{equation}\label{iota}f_*^i=(\iota\circ\rho\circ f)_*^i=(\iota\circ\rho|_{S_\ep}\circ F)_*^i=\iota_*^i\circ(\rho|_{S_\ep})_*^i\circ F_*^i .\end{equation}
Considering the long exact sequence 
\begin{equation}\label{longfiber}
\xymatrix{\pi_i(S^{m-n-1})\ar[r]&\pi_i(S_\ep)\ar[r]^{(\rho|_{S_\ep})_*^i}&\pi_i(g(N))\ar[r]&\pi_{i-1}(S^{m-n-1})
}
\end{equation}and the fact that $F:\Si\to S_\ep$ is a covering map, Theorem \ref{completeness} easily 
implies the following
\begin{corollary} \label{corollary_exact} Under the hypotheses of Theorem \ref{completeness} it holds that:
\begin{enumerate} [(a)]
\item If $m-n=1$ then $f_*^i:\pi_i(\Si)\to \pi_i(M)$ is a monomorphism for $i=1$ and an isomorphism for 
$i\ge 2$;
\item \label{monom2}If $m-n= 2$ then  $f_*^i:\pi_i(\Si)\to \pi_i(M)$ is a monomorphism  for $i=2$ and an isomorphism for $i\ge 3$;
\item If $m-n\ge 3$ then  $f_*^i:\pi_i(\Si)\to \pi_i(M)$ is a monomorphism  for $i=1$, an isomorphism for $2\le i\le m-n-2$
and an epimorphism for $i=m-n-1$.
\end{enumerate}
\end{corollary}

To prove Theorem \ref{completeness} we will need to use the following result which follows  from
the proof of Proposition 4.1 in \cite{mz}:

\begin{proposition} Let $M$ be a Riemannian manifold and $g:N\to M$
an immersion. Assume that for any point $p\in M$ there exists $q\in N$ such that the distance
$d(p,g(N))=d(p,g(q))$. Assume further that $\exp^\perp$ is defined on all $\m N_g$. Then $M$ is complete.
\end{proposition}
\begin{corollary} \label{mendzhou}
Let $g:N\to M$ be an embedding such that $\exp^\perp: \m N_g\to M$ is a diffeomorphism. If $N$ is compact then $M$ is complete.
\end{corollary}
\begin{proof} [Proof of Theorem \ref{completeness}]
We first prove that $M$ is complete and $N$ is compact. By Corollary \ref{mendzhou} it suffices
to prove that $N$ is compact.
We first consider the case $m-n=1$. By Lemma \ref{FG}-\ref{G} the map $G:\Si\to g(N)$, given
by $G=\rho\circ f$ is a local diffeomorphism. By compactness of $\Si$ and connectedness of $g(N)$ we have that $G(\Si)=g(N)$, hence $N$ is compact since
$g$ is an embedding.
Now we consider the case that $m-n\ge 2$.  By Lemma \ref{intersect} we obtain that $f(\Si)\cap g(N)=
\emptyset$. Thus by Lemma \ref{FG}-\ref{F} the map $F:\Si\to S_\ep$ given by $F(p)=\g_{f(p)}(\ep)$ is a local diffeomorphism. From the connectedness of $S_\ep$ and the compactness of $\Si$ we conclude that $F(\Si)=S_\ep$.
This implies that $N$ is compact since $S_\ep=F(\Si)$ is a bundle over $N$ with 
compact fibers.

To prove Item \ref{covering2} we assume that $m-n=1$. We obtain that $G=\rho\circ f :\Si\to g(N)$ is a covering map  from the following facts: the map $G$ is a  local diffeomorphism; the manifold $\Si$ is compact; $N$ is connected. Item \ref{covering2} is proved.

From now on we assume that $m-n\ge 2$. We already saw that $F:\Si\to S_\ep$ is a local diffeomorphism onto $S_\ep$, which is diffeomorphic to $\m N_g^1$. From the
compactness of $\Si$ and the connectedness of $S_\ep$ it follows that $F$ is a covering map. 
From Lemma \ref{intersect} we have that $f(\Si)\cap g(N)=\emptyset$. 

In order to prove Items \ref{mn>2}-\ref{embedding} and \ref{mn>2}-\ref{disjointunionsplit} assume that $f$ is an embedding. By Claim \ref{Fdiffeo} in the proof of Theorem \ref{gth} we have that $F$ is a diffeomorphism onto
its image $F(\Si)=S_\ep$. Now we use Item (\ref{isotopy}) of Theorem \ref{gth}, by taking account that 
$\rho^{-1}(\rho(f(\Si))=\rho^{-1}(g(N))=M$ to conclude that $F:\Si\to S_\ep$ is a diffeomorphism via an ambient isotopy. 
Item \ref{mn>2}-\ref{embedding} is proved. We use again Theorem \ref{gth} obtaining that $M-f(\Si)=A\sqcup B$ where the closure $\bar A$ is a tubular neighborhood of $g(N)$ with boundary $f(\Si)$, and $\bar B$ is an unbounded manifold with boundary $f(\Si)$. Since $f(\Si)$ and $g(N)$ are compact we conclude that $\bar A$ is compact.
Item \ref{mn>2}-\ref{disjointunionsplit} is proved.
The proof of Theorem \ref{completeness} is complete.
\end{proof}

Theorem \ref{tgmth2} follows from Remark \ref{bolton},
Lemma \ref{orthogon}, Theorem \ref{completeness} and the following
\begin{proposition} \label{12linha} Let $f:\Si^{m-1}\to M^m$ be an immersion of a closed manifold $\Si$ in a complete simply-connected Riemannian manifold $M$ and $g:N^n\to M^m$ a complete isometric embedding  without focal points. Assume that the unit speed geodesics of $M$ tangent to $f(\Si)$ do not intersect $g(N)$ orthogonally. Then $N$ is compact and the following conclusions hold:
\begin{enumerate}[(1)]
\item\label{m-nigual1}  if $m-n=1$ then $\Si$ is diffeomorphic to $N$, $f$ is an embedding and a homotopy equivalence;
\item\label{m-nigual2} if $m-n=2$ then $\pi_1(\Si)$ is cyclic;
\item\label{m-n>=3} if $m-n\geq 3$ then $f$ is an embedding, $\Si$ is simply-connected, and $F:f(\Si)\to S_\ep$ 
is a diffeomorphism via a natural ambient isotopy.
\end{enumerate}
\end{proposition}

\begin{proof}   From Remark \ref{bolton} we have that $\exp^\perp:\m N_g\to M$ is a diffeomorphism,
hence $N$ is simply-connected. 

To prove Item \ref{m-nigual1} we assume that $m-n=1$ and consider again the map $G=\rho\circ f:\Si\to g(N)$. We have from Theorem \ref{completeness} and 
the simply connectedness of $N$  that $G$ is a diffeomorphism onto $g(N)$, hence $f$ is an embedding. Trivially we have that $G$ is homotopic to $f$. The embedding $G$ 
is obviously a homotopy equivalence because $M$ is a vector bundle
over $g(N)$.
 Item \ref{m-nigual1} is proved.

Now assume that $m-n=2$. Since $g(N)$ is simply connected the long exact sequence (\ref{longfiber}) implies 
that $\pi_1(S_\ep)$ is cyclic. Since $F:\Si\to S_\ep$ is a covering map we have the $F_*^1$ injects 
$\pi_1(\Si)$ into  $\pi_1(S_\ep)$, hence it is also cyclic. Item \ref{m-nigual2} is proved. 

If $m-n\ge 3$ the same exact sequence (\ref{longfiber})  implies that $S_\ep$ is simply connected, hence $F$ is 
a diffeomorphism via a natural isotopy on $M$ (see Theorem \ref{gth}, Item \ref{isotopy}), hence $\Si$ is simply connected and $f$ is an embedding. 
 Proposition \ref{12linha} is proved.
\end{proof}

To prove the next theorem we will need the following topological lemma. Since we didn't find
it in the literature, we will present its simple proof for the sake of completeness.

\begin{lemma} \label{topology} Let $M$ be a connected metric space, and $A, B$ subsets of $M$ with closures and boundaries connected. Assume that $\bar A\cap B\not=\emptyset$, $A\cap \bar B\not=\emptyset$, $\p A\cap\p B=\emptyset$, $\bar A\not\subset B$ and $\bar B\not\subset A$. Then $M=\mbox{int}(A)\cup \mbox{int}(B)$, where $\mbox{int}(A)$ denotes the interior of $A$.
\end{lemma}
\begin{proof} By hypothesis we have that $\bar A\cap B\not=\emptyset$
and $\bar A\cap (M-B)\not=\emptyset$.
By connectedness of $\bar A$ we conclude that $\bar A\cap\p
B\not=\emptyset$. Since $\p A\cap \p B=\emptyset$
we obtain that
\begin{equation*}\mbox{int}(A)\cap\p B=\bar A\cap\p B\not=\emptyset .
\end{equation*}
Thus the set
\ $\mbox{int}(A)\cap\p B$ is an open and
closed nonempty subset of the connected set $\p B$, hence we have that
$\mbox{int}(A)\cap\p B=\p B$.
Thus we obtain that $\p B\subset \mbox{int}(A)$. Similarly we can
prove that $\p A\subset \mbox{int}(B)$.
We conclude that $$\bar A\cup \bar B= (\mbox{int}(A)\cup\p
A)\cup(\mbox{int}(B)\cup\p B)\subset\mbox{int}(A)\cup
\mbox{int}(B).$$ Thus we obtain that $\bar A\cup \bar B=A\cup
B=\mbox{int}(A)\cup
\mbox{int}(B)$. By the connectedness of $M$ we conclude that $M=\mbox{int}(A)\cup \mbox{int}(B)$.
\end{proof}

Theorem \ref{tgmth} follows from the next result together with Theorem B and Lemma \ref{orthogon}.
Theorems \ref{completeness}, \ref{completeness2} and Lemma \ref{orthogon} 
together imply  Theorem
\ref{covering}.

\begin{theorem}\label{completeness2} Let $\Si$ be a closed manifold with finite fundamental group. Let $f:\Si^{m-1}\to M^m$ be an immersion in a Riemannian manifold $M$. Let $g:N\to M$ be an immersion such $\exp^\perp:\m N_g\to M$ is a covering map. Assume that the unit speed geodesics of $M$ tangent to $f(\Si)$ do not intersect $g(N)$ orthogonally. Then $M$ is complete,  $N$ is compact with finite fundamental group
and the following conditions hold:
\begin{enumerate}[(1)]
\item\label{Ncompactafinitalinha} if $m-n=1$ then $\Si$ and $N$ have diffeomorphic universal covers, $f_*^1$ is 
a monomorphism and any 
lifting $\ti f:\ti\Si\to \ti M$ is an embedding and a homotopy equivalence;
\item\label{it3linha} if $m-n\ge 2$ then it holds that:
\begin{enumerate}[(a)]
\item \label{intersection} $f(\Si)\cap g(N)=\emptyset$;
\item \label{Mnoncompact} $M$ is noncompact with finite fundamental group;
\item\label{diffcover} $\Si$ and $\m N_g^1$ have diffeomorphic universal covers;
\item\label{it5linha} the homomorphism $\iota_*^1:\pi_1(f(\Si))\to \pi_1(M)$, induced by the inclusion $\iota:f(\Si)\to M$, is surjective;
\item \label{mndois} if $m-n=2$ then $M$ and $N$ have infinite second homotopy groups, $f_*^2:\pi_2(\Si)\to \pi_2(M)$ is 
a monomorphism and 
the quotient group $\pi_2(M)/\pi_2(\Si)=\Bbb Z$;
\end{enumerate}
\item\label{m-n3linha} if $m-n\geq 3$ then $f_*^1$ is a monomorphism, any 
lifting $\ti f:\ti\Si\to \ti M$ is an embedding and $\ti f(\ti\Si)$ is diffeomorphic to $S_\ep(\ti g(\ti N))$ via a natural 
ambient isotopy, hence $\ti f$ is $(m-n-1)-$connected.
  \end{enumerate}
\end{theorem}

%\item\label{it5linha} the homomorphism $\iota_*^i:\pi_i(f(\Si))\to \pi_i(M)$, induced by the inclusion $\iota:f(\Si)\to M$, is surjective for $1\le i\le m-n-1$;\item \label{mndois} if $m-n=2$ then $f_*^i:\pi_i(\Si)\to \pi_i(M)$ is a monomorphism  for $i=2$ and an isomorphism for $i\ge 3$;\item \label{mndois} if $m-n=2$ then $f_*^i:\pi_i(\Si)\to \pi_i(M)$ is a monomorphism  for $i=2$ and an isomorphism for $i\ge 3$;\item\label{it6linha} the map $f_*^i:\pi_i(\Si)\to \pi_i(M)$ is a monomorphism  for $i=1$, an isomorphism for $2\le i\le m-n-2$ and an epimorphism for $i=m-n-1$;if further $f$ is an embedding then $f$ is $(m-n-1)$-connected.

\begin{proof} We will first prove that $M$ is complete, and $N$ is compact with finite fundamental group.

\begin{equation}\label{lifting}
\xymatrix{
\ti\Si\ar[rrr]^{\tilde f}\ar[dd]^{\nu} &&& \ti M=\m N_{\breve g}\ar[d]^{\varphi}\ar@/_1.7pc/[dd]_P &&& \tilde{N}\ar[lll]_{\ti g}\ar[dd]^{\mu}\ar[llldd]_{\breve g} \\
                            &&&  \m N_g\ar[d]^{\exp^\perp_g}   &&&  \\
            \Si\ar[rrr]^f &&&  M                           &&&  N\ar[lll]_g}\\
\end{equation}

In fact, let $\mu:\tilde N\to N$ be the universal covering and consider the map $\breve g=g\circ\mu$. For $\ti z\in \ti N$ and $z=\mu(\ti z)$, set $W_z=(dg_z(T_zN))^\perp$ and $W_{\ti z}=(d\breve g_{\ti z}(T_{\ti z}\ti N))^\perp$. It is easy to see that $W_z=W_{\ti z}$ 
and that 
$\varphi:\m N_{\breve g}\to \m N_g$ given by $\varphi(\ti z, v)=(\mu(\ti z),v)$ is a covering map. The manifold  $\ti M=\m N_{\breve g}$ is simply-connected since it is strongly deformation
retracted to $\tilde N\times\{0\}$. Thus $P=\exp^\perp_g\circ\varphi:\ti M\to M$ is the universal covering of $M$. Consider
on $\ti M$ and on $\ti N$ the induced metrics by $P$ and $\mu$, respectively.

Set $\ti g:\ti N\to \ti M$ given by $\ti g(\ti z)=(\ti z,0)$. Notice that $(P\circ \ti g)(\ti z)=
\exp_g^\perp(\varphi(\ti z,0))=\exp_g^\perp(z,0)=g(z)=(g\circ\mu)(\ti z)=\breve g(\ti z)$, hence $\ti g$ is a
lifting of $\breve g$ and $g$. Set $\ti\g:[0,+\infty)\to \ti M$ given by
$\ti\g(t)=(\ti z, tv)$, where $v\in W_{\ti z}=W_z$. Set $\g=P\circ \ti\g$. Notice that $\g(t)=\exp^\perp_g(\mu(\ti z),tv)=\exp^\perp_g(z,tv)$, hence
$\g$ is the geodesic on $M$ with $\g(0)=g(z)$ and $\g'(0)=v$. Furthermore we have that $P(\ti N\times\{0\})=\exp^\perp_g(\mu(\ti N)\times\{0\})=g(N)$. Since $\g=P\circ\ti\g$ is orthogonal to $g(N)$ and
$P$ is a local isometry, we have that
$\ti \g$ is a geodesic on $\ti M$ which is orthogonal to $\ti N\times\{0\}$. Set $w=\ti\g'(0)$.
We have $v=\g'(0)=dP_{(\ti z, 0)}\ti \g'(0)=dP_{(\ti z, 0)}w$. Then it holds that 
\begin{equation}\label{equationbreve}\exp^\perp_{\ti g}(\ti z, w)=\ti\g(1)=(\ti z, v)=(\ti z, dP_{(\ti z, 0)}w),
\end{equation}
hence $\exp^\perp_{\ti g}$ is bijective. Since
by hypothesis $\exp^\perp$ is a covering map, we have that $g$ is free of focal points and hence so is $\ti g$. Thus the map $\exp^\perp_{\ti g}:\m N_{\ti g}\to \ti M$ is a local diffeomorphism.  Since it is also bijective we conclude that $\exp^\perp_{\ti g}$ is a diffeomorphism.

Consider the universal covering $\nu:\tilde\Si\to
\Si$ with induced metric. The map $f$ admits a lifting $\ti f: \ti \Si \to \ti M$.
Since $\Si$ is compact with finite fundamental group we have that $\ti \Si$ is compact.
Since we are using induced metrics we have that unit speed geodesics tangent to $\ti f(\ti \Si)$ do not intersect $\ti g(\ti N)$ orthogonally. Since $\exp^\perp_{\ti g}:\m N_{\ti g}\to\ti M$ is a diffeomorphism,
Theorem \ref{completeness} applies for $\ti f:\ti \Si\to \ti M$ and $\ti g:\ti N\to \ti M$. Thus we obtain that $\ti M$ is complete and
$\ti N$ is compact. Since $P$ and $\mu$ are locally isometric covering maps we conclude that $M$ is complete and $N$ is compact with finite fundamental group.

 Now we will prove that $f_*^1$ is a monomorphism if $m-n\not= 2$. Under this codimension hypothesis Proposition  \ref{12linha} says that $\ti f$ is an embedding. We take a continuous closed curve $\al:[0,1]\to \Si$ such that
$\be=f\circ\al$ is trivial in $\pi_1(M)$. Let $\ti \al:[0,1]\to\ti\Si$ be a lifting of $\al$.
Consider the curve $\ti \be=\ti f\circ\ti\al:[0,1]\to \ti M$. Notice that $P\circ\ti\be=
(P\circ \ti f)\circ\ti\al=f\circ(\nu\circ\ti\al)=f\circ\al=\be$, hence $\ti \be$ is a lifting
of $\be$. Since $\be$ is trivial in $\pi_1(M)$ it follows that $\ti \be:[0,1]\to \ti M$ is
a closed curve. From the equality $\ti \be=\ti f\circ\ti\al$ and the fact that $\ti f$ is injective we easily see that $\ti\al$ is a closed curve, hence
$\al$ is trivial in $\pi_1(\Si)$. This implies that $f_*^1$ is a monomorphism.
This fact and Proposition \ref{12linha} together imply Items \ref{Ncompactafinitalinha} and \ref{m-n3linha}.

From now on we assume that $m-n\ge 2$. Item \ref{it3linha}-\ref{intersection} follows directly from Lemma  \ref{intersect}. 
Now we will prove Item \ref{it3linha}-\ref{diffcover}. Since $\exp_{\ti g}^\perp:\m N_{\ti g}\to \ti M=\m N_{\breve g}$ 
is a diffeomorphism we have from (\ref{equationbreve}) that $\exp_{\ti g}^\perp:\m N_{\ti g}^1\to \m N_{\breve g}^1$ is a diffeomorphism. By using the covering map $\varphi:\m N_{\breve g}\to \m N_g$ given by 
$\varphi(\ti z, v)=(\mu(\ti z),v)$ we see that $\varphi|_{\m N_{\breve g}^1}:\m N_{\breve g}^1\to \m N_{g}^1$ is 
also a covering map, hence $\m N_{\ti g}^1$ covers $\m N_g^1$. From Theorem \ref{completeness} we know that 
$\ti \Si$ covers $\m N_{\ti g}^1$, hence it covers $\m N_g^1$. Item \ref{it3linha}-\ref{diffcover} is proved.

Now we begin the proof of Item \ref{it3linha}-\ref{Mnoncompact}.
Consider any liftings $\hat f:\ti\Si\to \ti M$ and $\hat g:\ti N\to \ti M$ of $f$ and $g$, respectively.
Since we are using induced metrics on $\ti M$ and $\ti N$ we have that $\hat g$ is an isometric immersion free of
focal points. We already proved that $\ti M$ is complete and $\ti N$ is compact, hence we
conclude by Remark \ref{bolton} that $\exp^\perp_{\hat g}:\m N_{\hat g}\to \ti M$ is a diffeomorphism. Thus for any
$x\in \ti M-\hat g(\ti N)$ there exists a unique unit speed geodesic $\ga_x=\g_{x,\hat g}:[0,\infty)\to \ti M$ that starts orthogonally from $\hat g(\ti N)$ with $\ga_x(d_x)=x$, where $d_x$ is the distance between $x$ and $\hat g(\ti N)$. Let $S_\ep=S_{\ep,\hat g}\subset \ti M$ be the boundary of the $\ep$-tubular neighborhood of $\hat g(\ti N)$. Since $\nu, P, \mu$ are locally isometric covering maps we have that the unit speed geodesics
tangent to $\hat f(\ti \Si)$ do not intersect $\hat g(\ti N)$ orthogonally. By Theorem \ref{completeness} we have that $\hat f(\ti \Si)\cap \hat g(\ti N)=\emptyset$ and the map
$F:\ti \Si\to S_\ep$ given by $F(p)=\ga_{\hat f(p)}(\ep)$ is a covering map. Thus, for all $z\in S_\ep$, the
geodesic $\ga_{z}$ only intersects $\hat f(\ti\Si)$ transversely and this occurs finitely many times. Let $t_{z,\hat g}$ be the time of the
first contact and $T_{z,\hat g}$ the time of the last contact between $\ga_{z}$ and $\hat f(\ti\Si)$.

We claim that $T_{z,\hat g}$ and $t_{z,\hat g}$ depend continuously on $z\in S_\ep$. In fact, since $F$ is
a covering map, for any $z_0\in S_\ep$, there
exists a neighborhood $V\subset S_\ep$ of $z_0$ such that $F^{-1}(V)$ is a disjoint union of
open sets $U_1,\cdots U_k$ satisfying that, for all $j$, the restriction $F|_{U_j}\to V$ is a diffeomorphism, hence $f|_{U_j}$ is
an embedding and, for any  $z\in V$, the geodesic $\ga_{z}$ intersects transversely $f(U_j)$ at a unique point $\ga_{z}(s_j(z))$.
By transversality we know that $s_j:V\to (0,+\infty)$ is a smooth function. Thus the maps $z\in V\mapsto
T_{z,\hat g}$ and $z\in V\mapsto t_{z,\hat g}$ are, respectively, maximum and minimum of
smooth functions,  hence we have that $T_{z,\hat g}$ and $t_{z,\hat g}$ depend continuously on $z$.

Define the sets
$$W_{\hat f,\hat g}=\{\ga_{z,\hat g}(t)\bigm|z\in S_\ep\mbox{ and }0\le t\le T_{z,\hat g}\},$$
and
$$w_{\hat f,\hat g}=\{\ga_{z,\hat g}(t)\bigm|z\in S_\ep\mbox{ and }0\le t\le t_{z,\hat g}\}.$$
Since $\exp^\perp_{\hat g}$ is a diffeomorphism, and $T_{z,\hat g}$ and $t_{z,\hat g}$ depend
continuously on $z$ we easily see that $W_{\hat f,\hat g}$ and $w_{\hat f,\hat g}$ are tubular neighborhoods of $\hat g(\ti N)$ with $C^0$-boundaries $\p (W_{\hat f,\hat g})$ and $\p (w_{\hat f,\hat g})$, respectively. Clearly we have that
\begin{equation} \label{boundarycontainedfsigma}
\p (w_{\hat f,\hat g}), \,\p (W_{\hat f,\hat g})\subset \hat f(\ti\Si),
\end{equation}
 and by definition of $W_{\hat f,\hat g}$ it is easy
to see that $\hat f(\ti \Si)\subset W_{\hat f,\hat g}$.

Fix liftings $\ti f:\ti \Si\to \ti M$ and $\ti g:\ti N\to \ti M$ of $f$ and $g$, respectively. We claim that
\begin{equation}\label{inclusion}
\mr g(\ti N)\subset W_{\ti f,\ti g}, \mbox{ for any lifting } \mr g:\ti N\to \ti M \mbox{ of } g.
\end{equation}
To prove this, assume by contradiction that $\mr g(\ti N)\not\subset W_{\ti f,\ti g}$, for some
lifting $\mr  g$. By Item \ref{m-nigual2} of Theorem \ref{completeness} we have that $\ti f(\ti\Si)\cap
\mr g(\ti N)=\emptyset$, hence $\p(W_{\ti f,\ti g})\cap \mr g(\ti N)=\emptyset$. Thus the connectedness of  $\mr g(\ti N)$ implies that 

\begin{equation}\label{forapo} \mr g(\ti N)\subset \bigl(\mbox{int}(w_{\ti f,\mr g})\cap (M-W_{\ti f,\ti g})\bigr).\end{equation}

To get the desired contradiction we first prove  that
$\p(w_{\ti f,\mr g})\subset\p(W_{\ti f,\ti g})$. To see this we take a point $x\in \p(w_{\ti f,\mr g})$.
Set $z=\g_{x,\mr g}(\ep)$ and $\eta=\g_{x,\mr g}=\ga_{z,\mr g}$. By definition of $w_{\ti f,\mr g}$ we have that $x=\eta(t_{z,\mr g})$. Notice that $\p(w_{\ti f,\mr g})\subset \ti f(\ti\Si)\subset W_{\ti f,\ti g}$, hence $x=\eta(t_{z,\mr g})\in W_{\ti f,\ti g}$ and
$\eta(0)\in \mr g(\ti N)\subset \ti M-W_{\ti f,\ti g}$ and thus there must exist $t_0\in [0,t_{z,\mr g}]$ such
that $\eta(t_0)\in \p(W_{\ti f,\ti g})$. From the fact that $\p(W_{\ti f,\ti g})\subset \ti f(\ti \Si)$ and the  definition of $t_{z,\mr g}$ we have that $\eta([0,t_{z,\mr g}))\cap \p W_{\ti f,\ti g}\subset \eta([0,t_{z,\mr g}))\cap \ti f(\ti \Si)=\emptyset$. Thus we obtain that $t_0=t_{z,\mr g}$, hence $x=\eta(t_0) \in \p(W_{\ti f,\ti g})$. Thus we
proved that $\p(w_{\ti f,\mr g})\subset\p(W_{\ti f,\ti g})$. 

Notice that $\p(w_{\ti f,\mr g})$ and
$\p(W_{\ti f,\ti g})$ are connected closed $C^0$-manifolds with the same dimension $m-1$. Thus from the inclusion
 $\p(w_{\ti f,\mr g})\subset \p(W_{\ti f,\ti g})$ we obtain that
\begin{equation}\label{equality}
\p(w_{\ti f,\mr g})=\p(W_{\ti f,\ti g}).
\end{equation}
We assert that
\begin{equation}\label{subset}
\mbox{int}(w_{\ti f,\mr g})\subset M-W_{\ti f,\ti g}.
\end{equation}
In fact, assume by contradiction
that there exists $z'\in \mbox{int}(w_{\ti f,\mr g})\cap W_{\ti f,\ti g}$. By using (\ref{forapo}) and 
the connectedness of $\mbox{int}(w_{\ti f,\mr g})$ we obtain that $\mbox{int}(w_{\ti f,\mr g})\cap \p(W_{\ti f,\ti g})\not=\emptyset$, which
contradicts (\ref{equality}). Thus, using (\ref{subset}), we obtain that
\begin{equation}\label{interior}
\mbox{int}(w_{\ti f,\mr g})\cap \mbox{int}(W_{\ti f,\ti g})=\emptyset.
\end{equation}
Thus (\ref{equality}) and (\ref{interior}) imply together that $(w_{\ti f,\mr g}\cup W_{\ti f,\ti g})$ is an $m$-dimensional closed $C^0$-manifold. We conclude that $w_{\ti f,\mr g}\cup W_{\ti f,\ti g}=\ti M$, which contradicts the fact that $\ti M$ is noncompact. This proves (\ref{inclusion}).

Now we will finish the proof of Item \ref{it3linha}-\ref{Mnoncompact}. Fix $q\in N$ and $\ti q\in \mu^{-1}(\{q\})\subset
\ti N$. By using the Fundamental Lifting Theorem and (\ref{inclusion}) we obtain that $$P^{-1}(\{g(q)\})=\{\mr g(\ti q)\bigm| \mr g:\ti N\to \ti M \mbox{ is a lifting of }g
\}\subset W_{\ti f,\ti g}\,.$$ From the compactness of $W_{\ti f,\ti g}$ and the fact that $P^{-1}(\{g(q)\})$
is a discrete set we conclude that the set $P^{-1}(\{g(q)\})$ is finite, hence $M$ has finite fundamental
group and must be noncompact, since $\ti M$ is noncompact. Item \ref{it3linha}-\ref{Mnoncompact} is
proved.

Now we will prove Item \ref{it3linha}-\ref{it5linha}. Fix liftings $\ti f:\ti\Si\to \ti M$ and $\ti g:\ti N\to \ti M$ of $f$ and $g$, respectively. 
Consider the inclusion map $\iota:f(\Si)\to M$. To prove  that $\iota_*^1$ is surjective it suffices to show that $P^{-1}(f(\Si))$ is path-connected. Consider the
group $\mbox{Aut}(P)$ of the automorphisms $\al:\ti M\to
\ti M$ of the covering map $P:\ti M\to M$. Since we are using induced metrics we have that $\al$ is
an isometry for any $\al\in \mbox{Aut}(P)$. By the Fundamental Lifting Theorem we have that
\begin{equation*}
P^{-1}(f(\Si))=\bigcup_{\al\in \mbox{Aut}(P)}(\al\circ\ti f)(\ti\Si).
\end{equation*}
Thus to prove that  $P^{-1}(f(\Si))$ is path-connected it suffices to show that
\begin{equation*} \ti f(\ti\Si)\cap (\al\circ\ti f)(\ti\Si)\not=\emptyset,\mbox{ for any }\al\in\mbox{Aut}(P),
\end{equation*}
which, by using (\ref{boundarycontainedfsigma}), follows from the next assertion:
\begin{equation} \label{W}\p(W_{\ti f,\ti g})\cap \p(W_{(\al\circ\ti f),(\al\circ\ti g)})\not=\emptyset,\mbox{ for any }\al\in\mbox{Aut}(P).
\end{equation}
We assume by contradiction that assertion (\ref{W}) is false. Then there exists $\al\in\mbox{Aut}(P)$ such that $\p(W_{\ti f,\ti g})\cap \p(W_{\mr f,\mr g})=\emptyset$ where $\mr f=\al\circ \ti f$ and
$\mr g=\al\circ \ti g$.
Since $\mr f$ and $\mr g$ are lifting maps of $f$ and $g$, respectively, we have by (\ref{inclusion}) that $\ti g(\ti N)\subset W_{\ti f,\ti g}\cap W_{\mr f,\mr g}$. Since $W_{\ti f,\ti g}$ and $W_{\mr f,\mr g}$ are compact submanifolds with connected boundaries and $\ti M$ is noncompact it follows from Lemma \ref{topology} that  either $W_{\ti f,\ti g}\subset W_{\mr f,\mr g}$ or $W_{\mr f,\mr g}\subset W_{\ti f,\ti g}$. Assume without loss of the generality that $W_{\ti f,\ti g}\subset W_{\mr f,\mr g}$. 
Since $\p W_{\ti f,\ti g}\cap\p W_{\mr f,\mr g}=\emptyset$ it holds that $\p W_{\ti f,\ti g}\subset \mbox{int}(W_{\mr f,\mr g})$. 
Furthermore we have $\mbox{int}(W_{\ti f,\ti g})\subset\mbox{int}(W_{\mr f,\mr g})$, hence $W_{\ti f,\ti g}\subset
\mbox{int}(W_{\mr f,\mr g})$.  Thus $\vol(W_{\ti f,\ti g})<\vol(W_{\mr f,\mr  g})$.

Since $\al$ is an isometry it is easy to see that
\begin{equation}
\al(W_{\ti f,\ti g})= W_{(\al\circ\ti f),(\al\circ\ti g)}= W_{\mr f,\mr g}.
\end{equation}
Thus we have that $\vol(W_{\ti f,\ti g})=\vol(W_{\mr f,\mr  g})$. This contradiction concludes
the proof of Item \ref{it3linha}-\ref{it5linha}. 

To prove Item \ref{it3linha}-\ref{mndois} we assume that $m-n=2$. We have that $\ti f_*^2:\pi_2(\ti \Si)\to \pi_2(\ti M)$ is 
monomorphism by Item \ref{monom2} of Corollary \ref{corollary_exact}, hence $f_*^2:\pi_2(\Si)\to \pi_2(M)$ is also 
a monomorphism.

 Set $S_\ep=S_{\ep,\ti g}$  
and consider the long exact sequence associated to the fibration $S^1\to S_\ep\to \ti g(\ti N)$: 
\begin{equation}\label{longfiberagain}
\xymatrix{0=\pi_2(S^{1})\ar[r]&\pi_2(S_\ep)\ar[r]^\psi&\pi_2(\ti g(\ti N))\ar[r]^\xi&\pi_{1}(S^{1})=\Bbb Z\ar[r]^\de&\pi_1(S_{\ep})\,.
}
\end{equation}
Since $\ti \Si$ is compact and covers $S_\ep$ we have that $\pi_1(S_{\ep})$ is finite.  The kernel $\mbox{Ker}(\de)$ is 
a subgroup of $\Bbb Z$ and the quotient $\Bbb Z/{\mbox{Ker}(\de)}$ is a subgroup of $\pi_1(S_{\ep})$, hence it is finite. 
This implies that $\xi(\pi_2(\ti g(\ti N))=\mbox{Ker}(\de)$ is cyclic infinite.  As a consequence we have that 
$\pi_2(N)$ and $\pi_2(M)$ are infinite since they are both isomorphic to $\pi_2(\ti g(\ti N))$. We also have that
\begin{eqnarray*}\Bbb Z&=&\mbox{Ker}(\de)=\xi(\pi_2(\ti g(\ti N)))=\pi_2(\ti g(\ti N))/\mbox{Ker}(\xi)
=\pi_2(\ti g(\ti N))/\psi(\pi_2(S_\ep))\\
&=&\pi_2(M)/\pi_2(\Si).
\end{eqnarray*}
In the last equality we used that 
$\pi_2(\ti g(\ti N))$ is isomorphic to $\pi_2(M)$ and that $\psi(\pi_2(S_\ep))$ is isomorphic to 
$\pi_2(\Si)$, since $\psi$ is a monomorphism and $\pi_2(\Si)$ is isomorphic to $\pi_2(S_\epsilon)$. 
Theorem \ref{completeness2} is proved.
\end{proof}

%Fix $2\le i\le m-n-1$ and a point $z\in \p(W_{\ti f,\ti g})$. Consider a continuous map $\psi:S^i\to M$ with $\psi(p)=P(z)$ for some point $p\in S^i$. We need to prove that there exists a homotopy $\psi_t$ with $\psi_t(p)=P(z)$ for all $t\in[0,1]$, $\psi_0=\psi$ and $\psi_1(S^i)\subset f(\Si)$. Since $i\ge 2$ there exists a lifting $\ti\psi:S^i\to \ti M$ of $\psi$ with $\ti\psi(p)=z$. Since $\p(W_{\ti f,\ti g})$ is the boundary of a tubular neighborhood of $\ti g(\ti N)$ we have that the inclusion map $j:\p(W_{\ti f,\ti g})\to \ti M$ is $(m-n-1)$-connected. Thus $\ti\psi$ admits a homotopy $\ti\psi_t$ with $\ti\psi_t(p)=z$ for all $t\in[0,1]$, $\ti\psi_0=\ti\psi$ and $\ti\psi_1(S^i)\subset \p(W_{\ti f,\ti g})$. Thus $\psi_t=P\circ\ti\psi_t$ is the desired homotopy, since $P(\p(W_{\ti f,\ti g}))\subset P(\ti f(\ti \Si))=f(\Si)$.

\section{\bf Examples}\label{secexamples}

Let $p:M\to N$ be a vector bundle, where $V_x$ denotes the fiber over $x$. It is well known that there
exists a smooth map $x\in N\mapsto \left<\,,\,\right>_x$, where $\left<\,,\,\right>_x$ is
an inner product on $V_x$. This map is usually called a fiber metric of $p$.

\begin{proposition} \label{bundle} Let $p:M^m\to N^n$ be a vector bundle over the manifold
$N$ and let  $V_x$ denote the fiber over $x$.
Fix a smooth fiber metric $x\in N\mapsto \left<\,,\,\right>_x$, where $\left<\,,\,\right>_x$ is
an inner product on $V_x$. Let $S_x\subset V_x$ be the unit sphere centered at the origin and
set $\m S_1=\cup_{x\in N}S_x$. Then there exists a Riemannian metric on the total space
$M$ such that the hypersurface
$\m S_1$ is totally geodesic and the null section $g:N\to N_0=g(N)\subset M$ satisfies  that
$\exp^\perp:\m N_g\to M$ is a diffeomorphism, hence $g$ is free of focal points. If further $N$ is compact then $M$ is
complete.
\end{proposition}

\begin{proof}  It is well known and easy to see that there exists a smooth distribution $z\in M\mapsto H_z\subset T_zM$
 where $H_z$ is an $n$-dimensional linear subspace which is transversal to the submanifold $V_{p(z)}$ at $z$ and
 has the property that  $H_{g(y)}=T_{g(y)}N_0$ for all $y\in N$. Fix $z\in M$ and set $p(z)=x$.
 We have a decomposition $T_zM=H_z\oplus
 T_z(V_x)$. Given $X\in T_zM$ we write $X=X_H+X_V$ with $X_H\in H_z$ and $X_V\in T_z(V_x)$.
 Given $X\in T_z(V_x)$, there exists a unique $K_z(X)\in V_x$ such that $X=\frac d{dt}\bigm|_{t=0}(z+tK_z(X))$.
It is easy to see that $dp_z\bigm|_{H_z}:H_z\to T_xN$ and $K_z:T_z(V_x)\to V_x$ are
isomorphisms. Fix any Riemannian metric $\om_N$ on $N$. We define on $M$ the Sasaki type
Riemannian metric $\om_1$ (compare with \cite{so}) by the equality
 \begin{equation}\label{sasakimetric} \om_1(X,Y)=\om_N(dp_z(X_H),dp_z(Y_H))+\left<K_z(X_V),K_z(Y_V)\right>_x,\end{equation}
 for any $X,Y\in T_zM$.

Since $H_{g(x)}=T_{g(x)}N_0$ for all $x\in N$ we conclude that $N_0=g(N)$ is orthogonal to the
submanifold $V_x$.
Thus we have that
\begin{equation*}\label{Ng} \m N_g=\m N_{g,\,\om_1}=\{(x,v)\bigm|x\in N,\ v\in (T_{g(x)}N_0)^\perp\}=
\{(x,v)\bigm|x\in N,\ v\in T_{g(x)}V_x\}.
\end{equation*}
Thus we have
a natural diffeomorphism $\varphi:\m N_g\to M$ given by $\varphi(x,v)=K_{g(x)}(v)$. In fact, it is easy to
see that the inverse map satisfies $\varphi^{-1}(z)=(p(z),(K_{g(p(z))})^{-1}(z))$. The map $\varphi$  is
in
fact an isomorphism between vector bundles.

We claim that $\m S_1=\varphi(\m N_g^1)$, where $\m N_g^1$ is the unit normal bundle
of $g$. In fact, by using (\ref{sasakimetric}), we obtain the following equivalences:
\begin{eqnarray*}(x,v)\in \m N_g^1 &\iff&
x\in N, \ v\in T_{g(x)}V_x\mbox{ and }\om_1(v,v)=1\\
&\iff&(x,v)\in\m N_g\mbox{ \ and\ }\left<K_{g(x)}(v),K_{g(x)}(v)\right>_{x}=1\\
&\iff&(x,v)\in\m N_g\mbox{ and }\left<\varphi(x,v),\varphi(x,v)\right>_x=1\\
&\iff&(x,v)\in\m N_g\mbox{ and }\varphi(x,v)\in S_x\\
&\iff&(x,v)\in\m N_g\mbox{ and }\varphi(x,v)\in \m S_1\\
&\iff&(x,v)\in \varphi^{-1}(\m S_1).
\end{eqnarray*}
We conclude that $\m N_g^1=\varphi^{-1}(\m S_1)$, hence $\m S_1=\varphi(\m N_g^1)$. Our claim is proved.

In fact it can be proved that the normal exponential map $\exp^\perp_{g,\,\om_1}:\m N_g\to M$, associated with the map $g$ and the metric
$\om_1$, coincides with the diffeomorphism $\varphi$. This can be done by making computations similar to those  for
the Sasaki metrics in $TN$ and $\m N_g$ (see \cite{do}, \cite{gk}, \cite{by}). To avoid these computations we will complete the
proof of Proposition \ref{bundle} without using this fact.
We will just consider on $M$ the metric $\om_2$ induced
by the diffeomorphism $\varphi$ from the Sasaki metric on $\m N_g$ (see \cite{by}).
Consider the geodesic $\bar \g:[0,+\infty)\to \m N_g$ given by $\bar \g(t)=
(x,tv)$. It is well known that $\bar \g'(0)$ is orthogonal to $N\times\{0\}$ with respect
 to the Sasaki metric. This implies that the map $\g=\varphi\circ\bar \g$ is an $\om_2$-geodesic which is $\om_2$-orthogonal to $\varphi(N\times\{0\})=K_{g(N)}(0)=g(N)=N_0$. Thus we have that $\exp_{g,\,\om_2}^\perp
(x,tv)=\g(t)=\varphi(x,tv)$, hence $\exp^\perp_{g,\,\om_2}=\varphi$, hence $\exp^\perp_{g,\,\om_2}$
is a diffeomorphism and even an isometric isomorphism between the vector bundles
$\m N_g$ and $M$. For this metric it is easy to see that
the normal bundle $\m N_{g,\,\om_2}$ coincides with $\m N_g$ and it holds that
\begin{equation}\label{w2}\exp^\perp_{g,\,\om_2}(x,tv)=\varphi(x,tv)=tK_{g(x)}(v).\end{equation}

We set:
$$\m S_r=\exp^\perp_{g,\,\om_2}(\{(x,v)\in \m N_g\bigm|\om_1(v,v)=r^2\}),$$
$$\m B_r=\exp^\perp_{g,\,\om_2}(\{(x,v)\in \m N_g\bigm|\om_1(v,v)<r^2\}).$$
Since $\exp^\perp_{g,\,\om_2}:\m N_g\to M$ is a diffeomorphism, it is easy to see that a geodesic
$\g:\real\to M$ satisfies
\begin{eqnarray}\label{orthogonality}\g\mbox{  is }
\om_2-\mbox{orthogonal to }N_0&\iff&\g\mbox{ is }\om_2-\mbox{orthogonal to }\m S_r\mbox{ for some }r>0\nonumber\\
&\iff&\g\mbox{ is }\om_2-\mbox{orthogonal to }\m S_r\mbox{ for any }r>0.
\end{eqnarray}
With respect to the metric
$\om_2$, by using (\ref{orthogonality}) it is easy
to see that if $\si_1,\si_2$ are geodesics orthogonal to $\m S_1$ then
they may be defined on all $\Bbb R$ (since they are reparameterizations of geodesics
orthogonal to $g(N)$) and we also have that
either
\begin{equation}\label{intersections} \si_1(\Bbb R)=\si_2(\Bbb R),\ \si_1(\Bbb R)\cap\si_2(\Bbb R)=\emptyset \ \mbox{ or }\ \si_1(\Bbb R)\cap\si_2(\Bbb R)=\{z\}\subset N_0.\end{equation}
Thus we conclude from (\ref{intersections}) that $M-N_0$ is an open tubular neighborhood of
$\m S_1$ in the general sense of Definition \ref{tubular}. More precisely, if we consider
 the inclusion map $\iota:\m S_1\to M$ we have that
$$\exp^\perp_{\iota,\,\om_2}|_W
:W\to M-N_0$$
 is a diffeomorphism, where $W=\{(z,t\nu(z))\in \m N_{\iota,\,\om_2}\bigm|t\in (-1, +\infty)\}$ and $\nu(z)\in T_zM$ is the $\om_2$-unitary vector orthogonal to
$\m S_1$ which points outwards the set $\m B_1$.

For $0<s<1$, let $U_s=U_{s,\,\om_2}$ denote an open $s$-tubular neighborhood of $\m S_1$ with
respect to the metric $\om_2$.
From (\ref{orthogonality}) and the facts that  $\exp^\perp_{g,\,\om_2}\m N_g\to M$ and $\exp^\perp_{\iota,\,\om_2}|_W:W\to M-N_0$ are
diffeomorphisms it is easy to see that
\begin{equation}\label{Us} \partial U_s=\m S_{1-s}\cup \m S_{1+s} \mbox{ for all }0<s<1.\end{equation}

Now we begin the construction of a metric on $M$ such that  $\m S_1$ becomes totally
geodesic.
We consider on $\m N_{\iota,\,\om_2}$ the Sasaki metric and introduce on $M-N_0$ the metric $\om_3$
induced by the diffeomorphism  $\exp^\perp_{\iota,\,\om_2}\bigm|_{W}$. Notice that $\m S_1$ is a
totally geodesic hypersurface
 of $M-N_0$ if we consider the metric $\om_3$. Furthermore we have trivially that, for any
 $0<s<1$ it holds that
 \begin{equation} \label{23}U_{s}=U_{s,\,\om_2}=U_{s,\,\om_3}.\end{equation}
 Fix $0<\de<\ep<1$.
 We consider a smooth bump function $\zeta:M\to [0,1]$ such that $\zeta(z)=1$ on $U_\de$ and $\zeta(z)=0$ outside $U_\ep$. We define on $M$ the metric
\begin{equation}\label{finalmetric} \om=(1-\zeta)\,\om_2+\zeta\,\om_3.\end{equation}
Thus the hypersurface $\m S_1$ is totally geodesic relatively to the metric $\om$. Furthermore,
take an $\om_2$-geodesic $\g$ which starts from $N_0$ orthogonally with
respect to $\om_2$. By (\ref{orthogonality}) we see that $\g$ is $\om_2$-orthogonal to $\m S_1$,
hence it is also $\om_3$-orthogonal to $\m S_1$ and it holds that $\om_2(\g'(t),\g'(t))
=\om_3(\g'(t),\g'(t))$, since the Sasaki metric $\om_3$ preserves orthogonality and length on radial directions. By (\ref{finalmetric}) we conclude that
\begin{equation}\label{gauss}\om(\g'(t),\g'(t))=\om_2(\g'(t),\g'(t))
=\om_3(\g'(t),\g'(t)).
\end{equation}

Our goal
now is to prove that $\exp^\perp_{g,\,\om}:\m N_g\to M$ is a diffeomorphism (in particular
$g:N\to (M,\om)$ will be free of focal points). For this it suffices to show that $\varphi=\exp^\perp_{g,\,\om}$. From the
construction of the metric $\om$  it follows that if $\om_2(v,v)<1-\ep$ then it holds that $\varphi(x,v)=\exp^\perp_{g,\,\om}(x,v)$.
Thus we just need to show that for all $(x,v)\in \m N_g-(N\times\{0\})$
it holds that the curve $\g=\g_{x,\,v}:[0,+\infty)\to V_x\subset M$ given by
$\g(t)=\varphi(x,tv)$ is an $\om$-geodesic (we already know that $\g$ is an
$\om_2$-geodesic by (\ref{w2})).

We fix an $\om_2$-geodesic $\g=\g_{x,\,v}$ as above with $\om_2(v,v)=1$.
We know that $\g|_{(0,+\infty)}$
is an $\om_3$-geodesic. By (\ref{Us}), (\ref{23}) and (\ref{finalmetric}) we just need to prove that $\g|_{([1-\ep,\,1-\de]\cup[1+\de,\,1+\ep])}$
is an $\om$-geodesic. We will prove that $\g|_{[1+\de,1+\ep]}$ is an $\om$-geodesic, since
the other case is similar. Fix $t\in [1+\de,1+\ep]$. By the Gauss Lemma we have that
\begin{equation} \label{om2}\om_2(\g'(t),\eta)=0 \iff \eta\in T_{\g(t)}\m S_t.\end{equation}
Set $t=1+s$ with $\de<s<\ep$. From the Gauss Lemma, we obtain from (\ref{Us}) and (\ref{23}) that
\begin{eqnarray}\label{om3} \om_3(\g'(t),\eta)=0 &\iff& \eta\in T_{\g(t)}\left(\partial U_{s,\,\om_3}\right)
=T_{\g(t)}(\m S_{1-s}\cup \m S_{1+s})\\&=& T_{\g(t)}(\m S_{2-t}\cup \m S_{t})=T_{\g(t)}\m S_{t}.\nonumber\end{eqnarray}
Thus from (\ref{finalmetric}), (\ref{om2}) and (\ref{om3})  it holds that if $\eta\in T_{\g(t)}\m S_{t}$
then $\om(\g'(t),\eta)=0$. Thus we have the inclusion $T_{\g(t)}\m S_t\subset\Omega_t=\{\eta\in T_{\g(t)}M\bigm|\om(\g'(t),\eta)=0\}$ between $(m-1)$-dimensional linear spaces, which implies
that $\Omega_t=T_{\g(t)}\m S_{t}$, hence it holds that
\begin{equation}\label{Omega} \om(\g'(t),\eta)=0\iff \om_2(\g'(t),\eta)=0\iff\om_3(\g'(t),\eta)=0.\end{equation}
Now we are in condition to prove that $\g$ is an $\om$-geodesic at $t\in[1+\de,1+\ep]$. On a small neighborhood of $\g(t)$
we consider the smooth $\om$-unitary vector field $X$ given by the equation $X(\g_{y,\,u}(s))=\g_{y,\,u}'(s)$ with $\om_2(u,u)=1$ and $y$ in a neighborhood of $x\in N$. Note that   (\ref{gauss}) implies that \begin{equation}\label{XX}\om(X,X)=\om_2(X,X)=\om_3(X,X)=1. \end{equation}
 Let $\nabla, \nabla^2$ and $\nabla^3$
denote, respectively, the Levi-Civita connections associated to $\om, \om_2$ and $\om_3$, respectively.
Since $\om(X,X)=1$ we have that $\om\left(\nabla_XX,X\right)=0$. Thus  we only need to
prove that $\om\left(\nabla_XX,\eta\right)=0$ at $\g(t)$, for any $\eta\in \Omega_t$. Fix
$\eta_0\in \Omega_t$ and extend it to a smooth vector field $\eta$ on a neighborhood of $\g(t)$.
From (\ref{XX}) we have that
\begin{equation}\label{eta2} \eta(\om(X,X))=\eta(\om_2(X,X))=\eta(\om_3(X,X))=0.\end{equation}
Thus
by using the Koszul formula and the fact that $\g|_{[1+\de,1+\ep]}$ is a geodesic with respect to $\om_2$
and $\om_3$ we have that
\begin{equation}\label{Koszul2}0=\om_2\left(\nabla^2_XX,\eta\right)=
X\left(\om_2\left(X,\eta\right)\right)-\om_2\left([X,\eta],X
\right),
\end{equation}
\begin{equation}\label{Koszul3}0=\om_3\left(\nabla^3_XX,\eta\right)=
X\left(\om_3\left(X,\eta\right)\right)-\om_3\left([X,\eta],X
\right).
\end{equation}
By using the Koszul Formula, it follows from (\ref{finalmetric}),
(\ref{Omega}), (\ref{eta2}), (\ref{Koszul2}), (\ref{Koszul3}) that:
\begin{eqnarray}
\om\left(\nabla_XX,\eta\right)&=&
X\left(\om\left(X,\eta\right)\right)-\om\left([X,\eta],X
\right)\\&=& X\bigl\{(1-\zeta)\om_2\left(X,\eta\right)+\zeta\om_3\left(X,\eta\right)\bigr\}\nonumber\\&-&
(1-\zeta)\om_2\left([X,\eta],X\right)-\zeta\om_3\left([X,\eta],X\right)\nonumber\\&=&
X(\zeta)\bigl\{\om_3(X,\eta)-\om_2(X,\eta)\bigr\}\nonumber\\ &+&
(1-\zeta)\bigl\{X\left(\om_2\left(X,\eta\right)\right)-\om_2\left([X,\eta],X
\right)\bigr\}\nonumber\\&+&
\zeta\,\bigl\{X\left(\om_3\left(X,\eta\right)\right)-\om_3\left([X,\eta],X
\right)\bigr\}\nonumber\nonumber\\&=&X(\zeta)\bigl\{\om_3(X,\eta)-\om_2(X,\eta)\bigr\}\nonumber\\&=&0.\nonumber
\end{eqnarray}
We conclude that $\g$ is an $\om$-geodesic, hence $\varphi=\exp^\perp_{g,\,\om}$ and therefore the map
$\exp^\perp_{g,\,\om}:\m N_g\to M$ is a diffeomorphism.

Finally, if we have the additional hypothesis that $N$ is compact, we obtain from Corollary  \ref{mendzhou}
that $M$ is complete. Proposition \ref{bundle} is proved.
\end{proof}

\begin{corollary} \label{diffeomorphism} Let $M,N$ be Riemannian manifolds and $g:N\to M$ an immersion such that $\exp^\perp_g: \m N_{g}\to M$ is
a diffeomorphism. Then there exists a Riemannian metric $\om$ on $M$ such that
 $\exp^\perp_{g,\,\om}: \m N_g\to M$  is a diffeomorphism and $\m S_1=\exp^\perp(\m N_g^1)\subset M$ is totally geodesic.
 \end{corollary}
 \begin{proof} Let $\iota:N\to \m N_g$ be the embedding given by $\iota(x)=(x,0)$.
 Apply Proposition \ref{bundle} to $\m N_g$ and obtain a metric $\om_1$ on $\m N_g$
 such that $\exp^\perp_{\iota,\,\om_1}:\m N_{\iota}\to \m N_g$ is a diffeomorphism
 and $\m N_g^1$ becomes totally geodesic. Now we just introduce the metric
 on $M$ induced by the diffeomorphism $\exp^\perp_g: \m N_g\to M$. Corollary \ref{diffeomorphism} is
 proved.
 \end{proof}

We recall some notations and facts about warped products. For a positive smooth function
$\rho:B\to \Bbb (0,\infty)$, the warped product $M=B\times_\rho N'$ is the product manifold $B\times N'$ with the metric
\begin{equation} \label{langle}\langle Z,\,W\rangle=ds^2\left(d\pi_B(Z),d\pi_B(W)\right)+(\rho\circ \pi_B)^2\langle d\pi_{N'}(Z),\,d\pi_{N'}(W)\rangle_{N'},\end{equation}
where $\pi_B,\pi_{N'}$ are the canonical projections of $B\times N'$ onto its corresponding factors and $\langle\ ,\,\rangle_{N'}$ is the metric of $N'$.\\

Consider the decomposition $TM=\m D(B)\oplus\m D(N')$, where $\m D(B)$ and $\m D(N')$ are the subbundles of $TM$ given by the distributions corresponding to the product foliations determined by $B$ and $N'$, respectively. For tangent vector fields $X$ on $B$ and $Y$ on $N'$  there exist unique lifts $X^h\in \m D(B)$ and $Y^v\in \m D(N')$ satisfying that $d\pi_B(X^h)=X$ and $d\pi_{N'}(Y^v)=Y$.   For tangent vector fields $X$ on $B$ and $Y$ on $N'$, the Levi-Civita connection $\bar\na$ of the warped product $M=B\times_\rho N'$ is related to the Levi-Civita connections $\na^B$ and $\na^{N'}$ of its corresponding factors by the following equations (see \cite{ON}):
\begin{enumerate}[(i)]
\item\label{horizontal} $\bar\na_{X^h}{X^h}=\left(\na^B_X X\right)^h$;
\item \label{misto}$\bar\na_{X^h}{Y^v}=\bar\na_{Y^v}{X^h}=ds^2(X,\eta)\,Y^v$;
\item \label{vertical}$\bar\na_{Y^v}{Y^v}= -\langle Y,Y\rangle_{N'}\ \eta^h+\left(\na^{N'}_YY\right)^v$;
\end{enumerate}
where $\eta=\na(\log \rho)$ and $\na$ is the gradient on $B$.

\begin{proposition}\label{warped} Take $B=(\mathbb{R}^k,ds^2)$ and
$ds^2=dr^2+\sigma^2(r)d\theta^2$ being the metric introduced in Example \ref{euclidean}.
Let $\m S=S^{k-1}\subset B$ be the totally geodesic unit sphere and $N'$ any manifold.
Consider a warped product $M=B\times_\rho N'$ and assume that
the gradient $(\nabla \rho)|_{\m S}$ is tangent to $\m S$. Then the inclusion maps:
$f:\Si=\m S\times N'\to M$ and $g:N=\{0\}\times N'\to M$  have the properties that $f$ is totally geodesic
and that the normal exponential map $\exp^\perp:\m N_g\to M$ is a diffeomorphism, hence $g$ is
free of focal points.
\end{proposition}

\begin{proof} 
By (\ref{langle}) the horizontal fibres $B_x=B\times \{x\}\subset M$ with the induced metric
are isometric to $B$. By
\ref{horizontal} above the fibers $B_x$ are totally geodesic submanifolds of $M$.
We know that the exponential map \ $\exp:T_0B\to B$ is a diffeomorphism, hence
$\exp:T_{(0,x)}B_x\to B_x$ is also a diffeomorphism for all $x$. Furthermore
 by (\ref{langle}) we have that $B_x$ is orthogonal to $N=\{0\}\times N'$, hence
 $\exp^\perp:\m N_g\to M$ is a diffeomorphism.

We claim that the inclusion map $f:\Si= \m S\times N'\to M$ is totally geodesic. In fact let $r:B\to [0,+\infty)$ be the distance function  from the
origin. Set $\mu=(\na r)^h$. We know that $\lan \mu,\mu\ran=ds^2(\na r,\na r)=1$. Since
$\na r$ is $B$-orthogonal to $\m S$ we have that $\mu$ is orthogonal to $\m S\times\{x\}$ for
any $x\in N'$.
We also have
that $\mu=(\na r)^h$ is orthogonal to $\{p\}\times N'$ for all $p\in B$,
hence we obtained a unitary vector field $\mu$ on $M-(\{0\}\times N')$ which is orthogonal
to $\Si$. Now fix $(p,x)\in \Si$ and $Z\in T_{(p,x)}\Si$. We claim that there exists a local tangent vector
field $\bar Z$ on $\Si$ extending $Z$ such that $\bar Z=X^h+Y^v$ for some tangent vector fields $X$ on $B$ and
$Y$ on $N'$. In fact, since $\Si=\m S\times N'$ we have a unique orthogonal decomposition
$Z=Z_{h}+Z_{v}$ with $Z_{h}\in T_{(p,x)}(\m S\times\{x\})$ and $Z_{v}\in
T_{(p,x)}(\{p\}\times N')$. Extend $d\pi_B(Z_h)$ to a tangent vector field $X$ on $B$ with 
the property that $X|_{\m S}$ is tangent to $\m S$ and extend 
 $d\pi_{N'}(Z_v)$ to a tangent vector field $Y$ on $N'$. Thus we have that $\bar Z=X^h+Y^v$ is
tangent to $\Si$ and $\bar Z(p,x)=Z$.
 By using \ref{horizontal}, \ref{misto} and \ref{vertical} above we obtain that
\begin{eqnarray*}
\lan{\bar\nabla}_{\bar Z} \bar Z,\mu\ran_{(p,x)}\!\!&=\!\!&
 \lan{\bar\nabla}_{X^h} X^h,(\na r)^h\ran+2\lan\bar\nabla_{X^h}{Y^v},(\na   r)^h\ran+\lan\bar\nabla_{Y^v}{Y^v},(\na r)^h \ran\\
 &=&ds^2 \left(\nabla^B_X X, \na r\right)_p+2\lan\lan X,\eta\ran Y^v,(\na r)^h\ran - 
 \lan \lan Y,Y\ran_{N'}\, \eta^h,(\na r)^h\ran\\
&=&ds^2 \left(\nabla^B_X X, \na r\right)_p - \lan Y,Y\ran_{N'}ds^2\left(\eta,\na r\right)_p.
\end{eqnarray*}

We have that $p$ belongs to the $ds^2$-totally geodesic submanifold $\m S\subset B$. 
By construction we have that $X|_{\m S}$ is tangent to $\m S$, hence it holds that $ds^2 \left(\nabla^B_X X, \na r\right)_p=0$.
Since $\eta|_{\m S}$ is by hypothesis tangent to $\m S$ and $\na r$ is $ds^2$-orthogonal 
to $\m S$ we have that $ds^2\left(\eta,\na r\right)_p=0$. Thus we conclude that $\left\langle \bar \nabla_{\bar Z}\bar Z,\mu \right\rangle_{(p,x)}=0$, hence $f$ is totally geodesic.
\end{proof}

\end{document}